\documentclass[reqno, 10pt]{amsart}
\usepackage{amsmath, amssymb}
\usepackage[margin=1in]{geometry}

\numberwithin{equation}{section}

\newtheorem{thm}{Theorem}[section]
 \newtheorem{lem}{Lemma}[section]
 \newtheorem{prop}{Proposition}[section]
 
\theoremstyle{remark}
\newtheorem{rem}{Remark}[section]
\newtheorem*{rem*}{Remark}
\theoremstyle{definition}

\newcommand{\qn}{\Pi_N}

\begin{document}

\title[Phase-field equations]{characterizing the stabilization size for semi-implicit
Fourier-spectral method to phase field equations}

\author[D. Li]{Dong Li}
\address[D. Li]
{Department of Mathematics, University of British Columbia, 1984 Mathematics Road,
Vancouver, BC, Canada V6T1Z2}
\email{dli@math.ubc.ca}

\author[Z. Qiao]{Zhonghua Qiao}
\address[Z. Qiao]
{Department of Applied Mathematics, The Hong Kong Polytechnic University, Hung Hom, Hong Kong}
\email{zhonghua.qiao@polyu.edu.hk}

\author[T. Tang]{Tao Tang}
\address[T. Tang]
{Department of Mathematics \& Institute for Computational and Theoretical Studies,
Hong Kong Baptist University, Kowloon, Hong Kong}
\email{ttang@math.hkbu.edu.hk}



\begin{abstract}
Recent results in the literature provide computational evidence that stabilized
semi-implicit time-stepping method can efficiently simulate phase field
problems involving fourth-order nonlinear diffusion, with typical examples
like the Cahn-Hilliard equation and the thin film type equation. The up-to-date
theoretical explanation of the numerical stability relies on the assumption that
the derivative of the nonlinear potential function satisfies a Lipschitz type
condition, which in a rigorous sense, implies the boundedness of the numerical solution.
In this work we remove the Lipschitz assumption on the nonlinearity and prove unconditional
energy stability for the stabilized semi-implicit time-stepping methods. It is shown
that the size of stabilization term depends on the initial energy and the perturbation
parameter but is independent of the time step. The corresponding error analysis is also established
under  minimal nonlinearity and regularity assumptions.

\end{abstract}
\maketitle
\section{Introduction}
\setcounter{equation}{0}

In this work we consider two phase field models: the Cahn-Hilliard (CH) equation and the molecular
beam epitaxy equation (MBE) with slope selection. The Cahn-Hilliard equation was originally
developed in \cite{CH58} to describe phase separation in a two-component system (such as metal alloy). It
typically takes the form
\begin{align} \label{1}
\begin{cases}
\partial_t u  =\Delta( - \nu \Delta u +f(u) ), \quad (x,t) \in \Omega\times(0,\infty),\\
u \Bigr|_{t=0} =u_0,
\end{cases}
\end{align}
where $u=u(x,t)$ is a real-valued function which represents the difference between two concentrations.
Due to
this fact the equation \eqref{1} is invariant under the sign change $u\to -u$.
{Another common form for CH is
\begin{align}
\begin{cases}
\partial_t u =\Delta w\\
w=-\epsilon \Delta u +\epsilon^{-1} f(u).
\end{cases}
\end{align}
As $\epsilon\to 0$ the chemical potential $w$ tends to
a limit which solves the two-phase Hele-Shaw (Mullins-Sekerka) problem (see
\cite{Pego89} for a heuristic derivation, \cite{ABC94} for a convergence proof (under the assumption
that classical solution to the limiting Hele-Shaw problem exists)).}
In \eqref{1} the spatial domain $\Omega$ is taken to be the usual $2\pi$-periodic
torus $\mathbb T^2= \mathbb R^2/2\pi \mathbb Z^2$. For simplicity we only consider the periodic case but our analysis can be generalized to other settings
(such as bounded domain with Neumann boundary conditions). The free energy term $f(u)$ is given by
\begin{align}
f(u)=F^{\prime}(u)=u^3-u, \qquad F(u)=\frac 14(u^2-1)^2.
\end{align}
The parameter $\nu>0$ is often called diffusion coefficient. Usually one is interested in the physical regime
$0<\nu \ll 1$ in which the dynamics of \eqref{1} is close to the limiting Hele-Shaw problem after some transient time.

For smooth solutions to \eqref{1}, the total mass is conserved:
\begin{align}
\frac d {dt} M(t)\equiv 0, \qquad M(t) = \int_{\Omega}u(x,t) dx.
\end{align}
In particular $M(t)\equiv 0$ if $M(0)=0$. Throughout this work we will only consider initial data $u_0$
with mean zero. On the Fourier side this implies the zeroth mode $\hat u(0)=0$. One can then
define fractional Laplacian $|\nabla|^s u$ for $s<0$ (see \eqref{def_fractional} for the definition of
$|\nabla|^s=(-\Delta)^{s/2}$). The energy functional associated with \eqref{1} is
\begin{align}
E(u) =\int_{\Omega} \left( \frac 12 \nu |\nabla u|^2 +F(u) \right) dx. \label{CH_energy}
\end{align}
As is well known,  \eqref{1} can be regarded as a gradient flow of $E(u)$ in $H^{-1}$. The basic energy
identity takes the form
\begin{align} \label{intro_-10}
\frac d {dt} E(u(t)) + \| |\nabla|^{-1} \partial_t u \|_2^2=0.
\end{align}
Note that $\partial_t u$ has mean zero and $|\nabla|^{-1} \partial_t u$ is well-defined. Alternatively
to avoid using $|\nabla|^{-1}$, one can rewrite \eqref{intro_-10} as
\begin{align}
\frac d {dt} E(u(t)) + \int_{\Omega}| \nabla( -\nu \Delta u + f(u) )|^2 dx =0.
\end{align}
It follows from the energy identity that
\begin{equation}\label{d1}
E(u(t)) \le E(u(s)), \quad \forall \;\; t\ge s.
\end{equation}
This gives a priori control
of $H^1$-norm of the solution. The global wellposedness of \eqref{1} is not an issue thanks to this fact.

There is by now an extensive literature on the numerical simulation
of the CH equation and related phase field models, see, e.g.,
\cite{BJL11,CNR10,CS98,CJPWW14,Du09,GHu11,HLT07, SY10,ZCST99} and the
references therein. On the analysis side,
it is noted that Feng and Prohl \cite{FP04} gave the error
analysis of a semidiscrete (in time) and fully discrete finite
element method for CH. Under a certain spectral assumption on the
linearized CH operator (more precisely, one has to assume
the existence of classical solutions to the corresponding Hele-Shaw
problem), they proved an error bound which depends on $1/\nu$
polynomially.

It is known that
explicit schemes usually suffer severe time-step restrictions and
generally do not obey energy conservation. To enforce the energy
decay property and increase the time step, a good alternative is to
use implicit-explicit (semi-implicit) schemes in which the linear
part is treated implicitly (such as backward differentiation in
time) and the nonlinear part is evaluated explicitly. For example,
in \cite{CS98} Chen and Shen considered the semi-implicit
Fourier-spectral scheme for \eqref{1} (set $\nu=1$)
\begin{align}
\frac{\widehat{u^{n+1}}(k)-\widehat{u^n}(k)} {\Delta t}= - |k|^4 \widehat{u^{n+1}}(k)
-|k|^2  \widehat{f(u^n)} (k),
\end{align}
where $\widehat{u^n}$ denotes the Fourier coefficient of $u$ at time step $t_n$.
 On the other hand, the semi-implicit schemes can generate large truncation errors. As a result
smaller time steps are usually required to guarantee accuracy and (energy)
stability. To resolve this issue, a class
of large time-stepping methods were proposed and analyzed in
\cite{FTY13,HLT07,SY10,XT06,ZCST99}. The basic
idea is to add an $O(\Delta t)$ stabilizing term to the numerical scheme to alleviate the time step constraint
whilst keeping energy stability. The choice of the $O(\Delta t)$ term is quite flexible. For example,
in \cite{ZCST99} the authors considered the Fourier spectral approximation of the modified CH-Cook
equation
\begin{align}
\partial_t C= \nabla \cdot \bigl( (1-a C^2) \nabla (C^3-C-\kappa \nabla^2 C ) \bigr).
\end{align}
The explicit Fourier spectral scheme is (see (16) in \cite{ZCST99})
\begin{align}
\frac {\widehat{C^{n+1}}(k,t)- \widehat{C^n}(k,t)} {\Delta t}
= ik \cdot \bigl\{ (1-a C^2) [ i k^{\prime} ( \{ -C+C^3 \}_{k^{\prime}}^n
+ \kappa |k^{\prime}|^2 \widehat{C^n}(k^{\prime},t)) ]_r \bigr\}_k.
\label{intro_300}
\end{align}
The time step for the above scheme has a severe constraint
\begin{align}
\Delta t \cdot \kappa \cdot K^4 \le 1,
\end{align}
where $K$ is the number of Fourier modes in each coordinate
direction. To increase
the allowed time step,
the authors of \cite{ZCST99}
added a term $-A k^4 ( \widehat{C^{n+1}} - \widehat{C^n})$ to the right-hand side
of \eqref{intro_300}. Note that on the real side, this term corresponds to a fourth order dissipation, i.e.
\begin{align*}
-A \Delta^2 (C^{n+1}-C^n)
\end{align*}
which roughly is of order $O(\Delta t)$.

In \cite{HLT07},
a stabilized semi-implicit scheme was considered for the CH model,
with the use of an order $O(\Delta t)$ stabilization term
\[
A\Delta(u^{n+1}-u^n).
\]
Under a condition on $A$ of the form
\begin{align} \label{bdd_earlier_2}
A \ge \max_{x\in \Omega} \Bigl\{ \frac 12 |u^n(x)|^2+\frac 14 |u^{n+1}(x)+u^n(x)|^2 \Bigr\}-\frac 12,
\qquad\forall\, n\ge0,
\end{align}
 one can obtain energy stability (\ref{d1}).
Note that the condition \eqref{bdd_earlier_2}
 depends nonlinearly on the numerical solution.
In other words, it implicitly uses the $L^\infty$-bound assumption
on $u^n$ in order to make $A$ a controllable constant.

  In \cite{SY10}, Shen and
Yang proved energy stability of semi-implicit schemes for the Allen-Cahn and the CH equations with truncated
nonlinear term. More precisely it is assumed that
\begin{align}
\max_{u\in \mathbb R} |f^{\prime}(u)| \le L
\end{align}
which is what we referred to as the Lipschitz assumption on the nonlinearity in the abstract. The same assumption was
adopted
recently in \cite{FTY13} to analyze stabilized Crank-Nicolson or
Adams-Bashforth scheme for both the Allen-Cahn and CH equations.

In a recent work  \cite{BJL11}, Bertozzi, Ju, and Lu considered a nonlinear diffusion model of the
form
\begin{align*}
\partial_t u = - \nabla \cdot ( f(u) \nabla \Delta u) + \nabla \cdot (g(u) \nabla u),
\end{align*}
where $g(u)=f(u) \phi^{\prime} (u)$, and $f$, $\phi$ are given smooth functions. In addition $f$ is assumed to be
non-negative. The numerical scheme considered in \cite{BJL11} takes the form
\begin{align}
\frac{u^{n+1}-u^n} {\Delta t}= -A \Delta^2 (u^{n+1}-u^n)
-\nabla \cdot (f (u^n) \nabla \Delta u^n) + \nabla \cdot( g(u^n) \nabla u^n),
\end{align}
where $A>0$ is a parameter to be taken large. One should note the striking similarity between this scheme
and the one introduced in \cite{ZCST99}. In particular in both papers the biharmonic stabilization of the
form $-A\Delta^2 (u^{n+1}-u^n)$ was used. The analysis in \cite{BJL11} is carried out under the additional
assumption that
\begin{align}
\sup_{n} \| f(u^n )\|_{\infty} \le A <\infty.
\end{align}
This is reminiscent of the $L^{\infty}$ bound on $u^n$.

Roughly speaking, all prior analytical developments are conditional in the sense
that either one makes a Lipschitz assumption on the nonlinearity, or one assumes certain a priori $L^{\infty}$
bounds on the numerical solution. It is very desirable to \emph{remove these technical restrictions} and establish
a more reasonable stability theory. Thus we consider the following.
Problem: {\em prove unconditional energy stability of large time-stepping semi-implicit numerical
schemes for general phase field models.}

Here unconditional means that no restrictive assumptions should be imposed on the time step. Of course one should
also develop the corresponding error analysis under minimal regularity and smoothness conditions.

The purpose of this work is to settle this problem for the spectral Galerkin case.
In a forthcoming work \cite{LQT14p1},
we shall
analyze the finite difference schemes for the CH model by using a completely different approach.

We now state our main results. We first consider a stabilized semi-implicit scheme introduced in
\cite{HLT07} following the earlier work \cite{XT06}. It takes the form
\begin{align} \label{semi_e1}
\begin{cases}
\displaystyle\frac{u^{n+1}-u^n}{\tau} = - \nu \Delta^2 u^{n+1} + A \Delta (u^{n+1}-u^n) +\Delta \qn ( f(u^n) ), \quad n\ge 0,
\\
u^0= \Pi_N u_0.
\end{cases}
\end{align}
where $\tau>0$ is the time step, and $A>0$ is the coefficient for the $O(\tau)$ regularization term.
For each integer $N\ge 2$, define
\begin{align*}
X_N= \operatorname{span}\Bigl\{ \cos (k\cdot x), \sin(k\cdot x):\quad k=(k_1,k_2)\in \mathbb Z^2, \; |k|_{\infty}=\max\{|k_1|, |k_2|\} \le N
\Bigr\}.
\end{align*}
Note that the space $X_N$ includes the constant function (by taking $k=0$).
The $L^2$ projection operator $\Pi_N: \, L^2(\Omega) \to X_N$ is defined by
\begin{align}
(\Pi_N u-u, \phi)=0, \qquad \forall\,\phi \in X_N,
\end{align}
where $(\cdot,\cdot)$ denotes the usual $L^2$ inner product on $\Omega$. In yet other words, the operator $\Pi_N$ is simply
the truncation of Fourier modes of $L^2$ functions to $|k|_{\infty}\le N$. Since  $\Pi_N u_0 \in X_N$,
by induction it is easy to check that  $u^n \in X_N$
for all $n\ge 0$. Note that one can recast \eqref{semi_e1} into the usual weak formulation, for example:
\begin{align*}
(d_t u^{n+1}, v) + A ( \nabla (u^{n+1}-u^n), \nabla v)
+ ( \nabla (f(u^n)), \nabla v) + \nu (\Delta u^{n+1}, \Delta v)=0,
\quad \forall\, v \in X_N,
\end{align*}
where $d_t u^{n+1}=(u^{n+1}-u^n)/\tau$.
However in our analysis it is more convenient to work with \eqref{semi_e1}.
Note that  $u^n$ has mean zero  for all $n\ge 0$ (since we assume $u_0$ has mean zero).

\begin{thm}[Unconditional energy stability for CH] \label{thm_es_1}
Consider \eqref{semi_e1} with $\nu>0$ and assume $u_0 \in H^2(\Omega)$
with mean zero.
Denote $E_0=E(u_0)$ the initial energy.
There exists a constant $\beta_c>0$ depending only on $E_0$ such that if
\begin{align} \label{A_bound_later}
A \ge \beta \cdot( \|u_0\|_{H^2}^2+\nu^{-1} |\log \nu|^2+1 ),\quad \beta\ge \beta_c,
\end{align}
then
\begin{align*}
E(u^{n+1}) \le E(u^n), \qquad\forall\, n\ge 0,
\end{align*}
where $E$ is defined by (\ref{CH_energy}).
\end{thm}

\begin{rem}
We stress that the above stability result works for any time step $\tau>0$. In particular the condition on the
parameter $A$
is independent of $\tau$. In order to keep the
argument simple, we do not try to optimize
the dependence of $A$ on the diffusion coefficient $\nu$ or the initial data $u_0$.
This can certainly be pushed further. For example, a close inspection of the proof of Theorem \ref{thm_es_1}
shows that it suffices to take $A$ such that
\begin{align*}
A \gg_{E_0} \|u_0\|_{\star}^2 + \nu^{-1} |\log \nu|^2+1,
\end{align*}
where
\begin{align*}
\|u_0 \|_{\star} = \sup_{N} \| \Pi_N u_0\|_{\infty}.
\end{align*}
The appearance of $\|u_0\|_{H^2}$ in \eqref{A_bound_later}
 is due to the embedding $\|u_0\|_{\star} \lesssim \| u_0\|_{H^2}$.  Alternatively one can replace
 the $H^2$-norm by weaker Besov norms. However we shall not dwell on this issue here further.
\end{rem}

\begin{rem}
One should note that in \eqref{A_bound_later}, the lower bound $\nu^{-1} |\log \nu|^2$ is formally consistent
with the predicted bound \eqref{bdd_earlier_2}. In terms of the PDE solution $u(t,x)$, the bound
\eqref{bdd_earlier_2} roughly asserts that
\begin{align*}
A \ge O( \| u(t)\|_{\infty}^2).
\end{align*}
For the PDE solution, there is no $L^{\infty}$ conservation and one has to trade it with the $\dot H^1(\mathbb T^2)$ (see \eqref{def_Hdots})
bound with some logarithmic correction. The energy conservation gives $\| u(t) \|_{H^1} \lesssim \nu^{-\frac 12}$,
and the log-correction gives $|\log (\nu)|$. Thus we need $A\gtrsim \nu^{-1} |\log \nu|^2$ from this heuristic argument.
 \end{rem}

There is an analogue of Theorem \ref{thm_es_1} for the MBE equation. The MBE equation
has the form
\begin{align} \label{mbe}
\begin{cases}
\partial_t h = -\nu \Delta^2 h + \nabla \cdot (g (\nabla h) ) , \qquad (x,t) \in \Omega \times (0,\infty), \\
h\Bigr|_{t=0} =h_0,
\end{cases}
\end{align}
where $h=h(x,t):\, \Omega \times \mathbb R \to \mathbb R$ represents the scaled height function of a thin film equation, and $g(z)=(|z|^2-1)z$ for
$z\in \mathbb R^2$. The domain $\Omega$ is again assumed to be the periodic torus $\mathbb T^2$. Equation
\eqref{mbe} can be regarded as an $L^2$ gradient flow of the energy functional
\begin{align}
E(h)= \frac {\nu}2 \| \Delta h\|_2^2 + \int_{\Omega} G(\nabla h) dx, \label{MBE_energy}
\end{align}
where $G(z)=\frac 14 (|z|^2-1)^2$ for $z \in \mathbb R^2$. Note the striking similarity between
the MBE energy \eqref{MBE_energy} and the CH energy \eqref{CH_energy}. Roughly speaking, $\nabla h$ is the
correct scaling analogue of $u$ in \eqref{1}. In fact it is well known that in one dimension the MBE equation can be
transformed into the CH equation through the change of variable $u=\partial_x h$. In recent \cite{LQT14} we
obtained new upper and lower gradient bounds for the MBE equation in dimensions $d\le 3$. A refined well-posedness
theory is also worked out there. Some of these results will be used in the $H^1$ error analysis in this work.
We refer to the introduction of \cite{LQT14} and also
\cite{EH66, BL13,BL14,LL03,L13,SWWW12,WWW10} for some background material and related well-posedness/ill-posedness results.


Consider the following semi-implicit scheme for MBE:
\begin{align} \label{semi_e2}
\begin{cases}
\displaystyle\frac{h^{n+1}-h^n}{\tau} = - \nu \Delta^2 h^{n+1} + A \Delta (h^{n+1}-h^n) + \qn \nabla \cdot ( g(\nabla h^n) ),
\quad n\ge 0, \\
h^0 = \Pi_N h_0.
\end{cases}
\end{align}
This scheme was introduced and analyzed in \cite{XT06} (see also
\cite{QZT11}). The authors of \cite{XT06} first introduced the
stabilized $O(\Delta t)$ term of the form $A \Delta (h^{n+1}-h^n)$
as given in (\ref{semi_e2}). They also proved that
the energy stability (\ref{d1}) under the condition
\begin{align} \label{bdd_earlier_1}
A \ge \frac 12 \| \nabla h^n \|_{\infty}^2 + \frac 14 \| \nabla (h^{n+1}+h^n) \|_{\infty}^2 -\frac 12, \qquad\forall
\, n\ge 0.
\end{align}
Again, it is seen that $A$ depends implicitly on the $L^\infty$ bound on the
numerical solution $h^n$.

The result below will provide a clean description on the size of the constant
$A$, in the sense that  $A$ is independent of  the $L^\infty$ bound on
the numerical solution.

\begin{thm}[Unconditional energy stability for MBE] \label{thm_es_2}
Consider \eqref{semi_e2} with $\nu>0$. Assume $h_0 \in H^3(\Omega)$ with mean zero.
There exists a constant $\beta_c>0$ depending only on $E_0$  such that if
\begin{align*}
A \ge \beta \cdot (\| h_0\|_{H^3}^2+\nu^{-1} |\log \nu|^2+1),\quad \beta\ge \beta_c,
\end{align*}
then
\begin{align*}
E(h^{n+1}) \le E(h^n), \qquad\forall\, n\ge 0,
\end{align*}
where $E$ is defined by (\ref{MBE_energy}).
\end{thm}

We now state the results for error estimates. We start with the CH equation.

\begin{thm}[$L^2$ error estimate for CH] \label{thm_CH_L2}
Let $\nu>0$. Let $u_0 \in H^s$, $s\ge 4$ with mean zero. Let $u(t)$ be the solution to \eqref{1} with initial
data $u_0$. Let $u^n$ be defined according to \eqref{semi_e1} with initial data $\Pi_N u_0$. Assume
$A$ satisfies  the same condition in Theorem \ref{thm_es_1}.
Define $t_m= m \tau$, $m\ge 1$. Then
\begin{align*}
\| u(t_m) - u^m \|_2
\le A\cdot e^{C_1 t_m} \cdot C_2 \cdot (N^{-s} +\tau).
\end{align*}
Here $C_1>0$ depends only on $(u_0,\nu)$, $C_2>0$ depends on $(u_0,\nu,s)$.
\end{thm}

For the MBE equation, we have the following $H^1$ error estimate. Note that due to the use of $H^1$
space the error bound below involves $N^{-(s-1)}$ instead of $N^{-s}$.

\begin{thm}[$H^1$ error estimate for MBE] \label{thm_MBE_L2}
Let $\nu>0$ and $h_0 \in H^s$, $s\ge 5$ with mean zero. Let $h(t)$ be the solution to the MBE equation with
initial data $h_0$. Let $h^n$ be defined according to \eqref{semi_e2} with initial data $\Pi_N h_0$.
Assume $A$ satisfies the same condition as in  Theorem \ref{thm_es_2}.
Define $t_m=m \tau$, $m\ge 1$. Then
\begin{align*}
\| \nabla (h(t_m) -h^m)\|_2 \le A \cdot e^{C_1 t_m} \cdot C_2 \cdot (N^{-(s-1)} +\tau),
\end{align*}
where $C_1>0$ depends on $(h_0,\nu)$, $C_2>0$ depends on $(\nu,h_0,s)$.
\end{thm}

\begin{rem}
On the one hand, the parameter $A$ in the added second order damping term has to be taken
sufficiently large to guarantee stability as was shown in Theorem \ref{thm_es_1} and Theorem \ref{thm_es_2}.
On the other hand, from the above error analysis, it is evident that the introduced damping term slows
down the error convergence rate which now depends linearly on the parameter $A$. In numerical practice the value of $A$
needs to be chosen judiciously so as to achieve relatively fast convergence while not losing stability. In yet other words
there exists a delicate ``balance" between stability and convergence.
\end{rem}

We end this section by introducing some notation and preliminaries used in this paper.

We shall use  $X+$ to denote $X+\epsilon$ for arbitrarily small $\epsilon>0$. Similarly we can
define $X-$. We denote by $\mathbb T^d= \mathbb R^d / 2\pi \mathbb Z^d$ the $2\pi$-periodic torus.

 Let $\Omega=\mathbb T^d$. For any function $f:\; \Omega\to
\mathbb R$, we use $\|f\|_{L^p}=\|f\|_{L^p(\Omega)}$ or sometimes $\|f\|_p$ to denote
the  usual Lebesgue $L^p$ norm  for $1 \le p \le
\infty$. If $f=f(x,y): \, \Omega_1\times\Omega_2 \to \mathbb R$, we shall denote by
$\|f\|_{L_x^{p_1} L_y^{p_2}}$ the mixed norm:
\begin{align*}
\|f\|_{L_x^{p_1} L_y^{p_2}} = \Bigl\|  \|f(x, y)\|_{L_y^{p_2}(\Omega_2)} \Bigr\|_{L_x^{p_1}(\Omega_1)}.
\end{align*}
In a similar way one can define other mixed norms such as $\|f \|_{C_t^0 H_x^m}$ etc.

   For any two quantities $X$ and $Y$, we denote $X \lesssim Y$ if
$X \le C Y$ for some constant $C>0$. Similarly $X \gtrsim Y$ if $X
\ge CY$ for some $C>0$. We denote $X \sim Y$ if $X\lesssim Y$ and $Y
\lesssim X$. The dependence of the constant $C$ on
other parameters or constants is usually clear from the context and
we will often suppress  this dependence. We denote $X \lesssim_{Z_1,\cdots,Z_m} Y$ if
$X\le C Y$, where the constant $C$ depends on the parameters $Z_1,\cdots,Z_m$.

We use the following convention for Fourier expansion on $\Omega=\mathbb T^d$:
\begin{align*}
 &(\mathcal F f)(k)=\hat f (k) = \int_{\Omega} f(x) e^{-i x\cdot k} dx,  \quad
 f(x) = \frac 1 {(2\pi)^d} \sum_{k \in \mathbb Z^d} \hat f(k) e^{ik\cdot x}.
\end{align*}

For $f:\, \mathbb T^d\to \mathbb R$ and $s\ge 0$, we define the $H^s$-norm and $\dot H^s$-norm of $f$ as
\begin{align} \label{def_Hdots}
&\| f \|_{H^s}= (2\pi)^{-\frac d2} \Bigl( \sum_{k\in \mathbb Z^d} (1+|k|^{2s}) |\hat f(k) |^2 \Bigr)^{\frac 12}, \quad
\| f \|_{\dot H^s}= (2\pi)^{-\frac d2} \Bigl( \sum_{k\in \mathbb Z^d} |k|^{2s} |\hat f(k) |^2 \Bigr)^{\frac 12},
\end{align}
provided of course the above sums are finite. Note that for $s=1$
\begin{align*}
\|f\|_{\dot H^1} = \| \nabla f \|_2.
\end{align*}

If $f$ has mean zero, then $\hat f(0) =0$ and in this
case
\begin{align*}
\|f \|_{H^s} \sim
\Bigl( \sum_{k\in \mathbb Z^d} |k|^{2s} |\hat f(k) |^2 \Bigr)^{\frac 12}.
\end{align*}
For $f$ with mean zero, one can also define its $\dot H^s$-norm for $s<0$ via
\begin{align*}
\|f \|_{\dot H^s} = (2\pi)^{-\frac d2} \Bigl(\sum_{0 \ne k\in \mathbb Z^d} |k|^{2s} |\hat f (k)|^2 \Bigr)^{\frac 12}
\end{align*}
provided the sum converges.

For mean zero functions, we can define the fractional Laplacian $|\nabla|^s$, $s \in \mathbb R$ via
the relation
\begin{align} \label{def_fractional}
\widehat{|\nabla|^s f} (k) =|k|^s \hat f(k), \qquad 0\ne k\in \mathbb Z^d.
\end{align}
The mean zero condition is only needed for $s<0$. Note that in accordance with the usual notation
we have $|\nabla|^s = (-\Delta)^{s/2}$. For any $s\in \mathbb R$,  we will use the notation $\langle \nabla \rangle^s
=(1-\Delta)^{s/2}$ which corresponds to the multiplier $(1+|k|^2)^{s/2}$ on the Fourier side.

We shall use the following simple interpolation inequality.
\begin{lem} \label{lem_simple_interp1}
For any $f\in \dot H^{-1}(\mathbb T^d) \cap \dot H^1(\mathbb T^d)$, we
have
\begin{align} \notag
\|f \|_2\le \| |\nabla|^{-1} f\|_2^{\frac 12} \| \nabla f\|_2^{\frac 12}.
\end{align}
Similarly for any $f \in L^2(\mathbb T^d) \cap \dot H^2 (\mathbb T^d)$, we have
\begin{align} \notag
\| \nabla f \|_2 \le \| f \|_2^{\frac 12} \| \Delta f \|_2^{\frac 12}.
\end{align}
\end{lem}
\begin{proof}
For the first inequality, note that $f$ has mean zero by assumption. Then by Plancherel we can write
\begin{align*}
\int f^2 dx = \int |\nabla| f \cdot |\nabla|^{-1} f dx.
\end{align*}
The result then follows from the Cauchy-Schwartz inequality. Note that $\| |\nabla| f\|_2 = \| \nabla f \|_2$.
The proof of the second inequality is even easier since
\begin{align*}
\int \nabla f \cdot \nabla f dx = -\int f \Delta f dx.
\end{align*}
\end{proof}

Occasionally we will need to use the Littlewood--Paley frequency projection
operators. To fix the notation, let $\phi_0 \in
C_c^\infty(\mathbb{R}^d )$ and satisfy
\begin{equation}\nonumber
0 \leq \phi_0 \leq 1,\quad \phi_0(\xi) = 1\ {\text{ for}}\ |\xi| \leq
1,\quad \phi_0(\xi) = 0\ {\text{ for}}\ |\xi| \geq 2.
\end{equation}
Let $\phi(\xi):= \phi_0(\xi) - \phi_0(2\xi)$ which is supported in $1/2\le |\xi| \le 2$.
For any $f \in \mathcal S^{\prime}(\mathbb R^d)$, $j \in \mathbb Z$, define
\begin{align*}
 &\widehat{\Delta_j f} (\xi) = \phi(2^{-j} \xi) \hat f(\xi), \quad
 \widehat{S_j f} (\xi) = \phi_0(2^{-j} \xi) \hat f(\xi), \qquad \xi \in \mathbb R^d.
\end{align*}

 Let $f:\, \mathbb T^d \to \mathbb R$ be
a smooth function. Note that $f$ can be regarded as a tempered distribution on $\mathbb R^d$ for which $\Delta_j f$
can be defined as above. For any $1\le p\le q\le \infty$, we recall
the following Bernstein inequalities (see \cite{LQT14} for a standard proof)
\begin{align}
&\| |\nabla|^s \Delta_j f \|_{L^p(\mathbb T^d)} \sim 2^{js}  \| \Delta_j f \|_{L^p(\mathbb T^d)}, \qquad s \in \mathbb R;
\label{b_per_e1}\\
& \|\Delta_j f \|_{L^q(\mathbb T^d)} \lesssim 2^{jd(\frac 1p-\frac 1q)} \|f\|_{L^p(\mathbb T^d)},
\qquad\, j \in \mathbb Z;
\label{b_per_e2}\\
& \|  S_j f \|_{L^q(\mathbb T^d)}
 \lesssim 2^{j d( \frac 1p - \frac 1 q)} \| f \|_{L^p(\mathbb T^d)}, \qquad j\ge -2. \label{b_per_e3}
\end{align}

In later sections, we will use (sometimes without explicit mentioning)
 the following interpolation inequality on $\mathbb T^2$: for $s>1$ and
 any $f\in H^s(\mathbb T^2)$ with mean zero, we have
 \begin{align} \label{log212}
 \| f\|_{L^{\infty}(\mathbb T^2)} \le 1+ C_s \|f\|_{\dot H^1(\mathbb T^2)} \log(3+\|f\|_{H^s(\mathbb T^2)}),
 \end{align}
 where $C_s>0$ is a constant depending only on $s$.
 \begin{rem}
The constant $1$ in the above inequality can be replaced by any other positive constants (with different corresponding constant $C_s$).
The mean zero condition is certainly needed in view of the $\|f\|_{\dot H^1}$ term on the RHS. If it is replaced
by $\|f\|_{H^1}$ then the inequality holds for any $f$ not necessarily with mean zero.
 \end{rem}
 We include a proof of \eqref{log212} for the sake of completeness.
 Since $f$ has mean zero we have $\Delta_j f=0$ for $j<-2$.
Let $j_0 \in \mathbb Z$ whose value will be chosen later. By
using the Bernstein inequality, we have
\begin{align*}
\|f\|_{L^{\infty}(\mathbb T^2)}
& \lesssim \sum_{-2\le j\le j_0} 2^j \| \Delta_j f\|_{L^2(\mathbb T^2)}
+ \sum_{j>j_0} 2^j 2^{-js} \| f\|_{H^s(\mathbb T^2)}, \notag \\
& \lesssim_s (j_0+3) \|f\|_{\dot H^1} + 2^{-j_0(s-1)} \| f\|_{H^s}.
\end{align*}
Choosing $j_0 = \operatorname{const} \cdot \log (3+\|f\|_{H^s})$ then yields \eqref{log212}.

We will need to use the usual  Sobolev embedding on $\mathbb T^d$. We include the precise
statement and also a proof here for the sake of completeness.
\begin{lem}[Sobolev embedding] \label{lem_sobo}
Let $d\ge 1$ and $0<s<d$. Then for any $\infty>p>\frac d{d-s}$, we have
\begin{align*}
\| \langle \nabla \rangle^{-s} f \|_{L^p(\mathbb T^d)} \lesssim_{s,p,d} \| f\|_{L^q(\mathbb T^d)}, \quad
\text{where }\frac 1 q=\frac 1p +\frac sd.
\end{align*}
\end{lem}
\begin{proof}
We shall write $\lesssim_{s,p,d}$ as $\lesssim$.
First note that the average of $f$ on $\mathbb T^d$ is easily bounded by $\|f \|_q$.  Thus we can assume
that $f$ has mean zero, this would imply $\Delta_j f =0$ for $j<-2$.
For convenience we may also assume $\|f \|_q=1$.
  Now let $j_0 $ be an integer whose value
will be chosen later.  For the low frequency piece we have
 \begin{align*}
|(\langle \nabla \rangle^{-s} S_{j_0} f) |(x)
& \lesssim \sum_{j=-2}^{j_0} 2^{-js} 2^{jd/q}
\| \Delta_j f \|_{q} \notag \\
& = \sum_{j=-2}^{j_0} 2^{-js} 2^{jd(\frac 1 p+\frac sd)}
\| \Delta_j f \|_{q} \notag \\
& \lesssim 2^{j_0 \frac d p} \| f\|_{q} = 2^{j_0 \frac dp}.
\end{align*}
For the high frequency piece, we have
\begin{align*}
\sum_{j>j_0} | (\langle \nabla \rangle^{-s}  \Delta_j f )(x)| \lesssim 2^{-j_0s}  (\mathcal M f)(x),
\end{align*}
where $\mathcal Mf$ is the maximal function (adapted to the periodic case, one can restrict to balls of size less than
$2\pi$ centered at the point $x$).   If $(\mathcal M f)(x) \lesssim 1$, we choose
$j_0=1$. If $(\mathcal Mf )(x) \gg 1$, then we choose $j_0$ such that
\begin{align*}
2^{j_0 (s+\frac dp)} \sim \mathcal Mf(x).
\end{align*}
Thus
\begin{align*}
| \langle \nabla \rangle^{-s} f |(x) \lesssim  (1+\mathcal Mf (x) )^{\frac  {d/p} { d/p+s}} \lesssim 1+ (\mathcal M f(x) )^{\frac q p}.
\end{align*}
This in turn implies the desired inequality.
\end{proof}

\section{Proof of Stability results}
\setcounter{equation}{0}

In this section, we will provide rigorous proofs for the stability results, i.e., Theorems \ref{thm_es_1} and
\ref{thm_es_2}.
\subsection{Proof of Theorem \ref{thm_es_1}}
Rewrite \eqref{semi_e1} as
\begin{align}
u^{n+1} = \frac{1-A \tau \Delta} {1+\nu \tau \Delta^2-A\tau \Delta} u^n
+ \frac{\tau \Delta \qn}{1+\nu \tau \Delta^2 -A\tau \Delta}  f(u^n).
\end{align}

\begin{lem} \label{z1}
There is an absolute constant $c_1>0$ such that for any $n\ge 0$,
\begin{align}
& \| u^{n+1} \|_{H^{\frac 32}(\mathbb T^2)} \le c_1 \cdot \Bigl(\frac{A+1}{\nu}+\frac 1 {A\tau} \Bigr) \cdot (E_n+1), \\
& \| u^{n+1} \|_{\dot H^1(\mathbb T^2)}
\le \Bigl(1+ \frac {1} A + \frac{3}A \|u^n\|_{\infty}^2 \Bigr)\cdot \|u^n\|_{\dot H^1(\mathbb T^2)},
\end{align}
where $E_n=E(u^n)$.
\end{lem}

\begin{proof}
In this proof for any two quantities $X$ and $Y$,  we shall use the notation $X\lesssim Y$ to
denote $X\le CY$ where $C>0$ is an absolute constant. For any $s \in \mathbb R$, we denote $\langle \nabla \rangle^s = (1-\Delta)^{s/2}$ which corresponds
to the multiplier $(1+|k|^2)^{s/2}$ on the Fourier side.

First note that on the Fourier side, we have for each $ 0\ne k \in \mathbb Z^d$,
\begin{align}
& \frac {(1+ A\tau |k|^2) |k|^{\frac 32} } { 1+ \nu \tau |k|^4 + A \tau |k|^2} \lesssim \frac 1 {A\tau} + \frac A {\nu}, \notag \\
&  \frac {\tau |k|^2 \cdot |k|^{\frac 32} }{1+\nu \tau |k|^4 + A\tau |k|^2}
 \lesssim  \frac 1 {\nu} |k|^{-\frac 12}. \notag 
\end{align}

Thus
\begin{align*}
\| u^{n+1} \|_{H^{\frac 32}}
&\lesssim \Bigl(\frac A {\nu} + \frac 1{A\tau} \Bigr) \| u^n \|_2 + \frac 1 {\nu}
\| \langle \nabla \rangle^{-\frac 12} \bigl( f(u^n)  \bigr) \|_2 \notag \\
&\lesssim \Bigl(\frac A {\nu}+\frac{1}{A\tau} \Bigr) \|u^n\|_2 + \frac 1 {\nu} \| (u^n)^3-u^n   \|_{\frac 43} \notag \\
& \lesssim \Bigl(\frac {A+1}  {\nu} + \frac 1 {A\tau} \Bigr) (E_n+1).
\end{align*}
In the second inequality above we have used the Sobolev embedding $\| \langle \nabla \rangle^{-1/2} h \|_{L^2(\mathbb T^2)}
\lesssim \| h \|_{L^{4/3}(\mathbb T^2)}$ (see Lemma \ref{lem_sobo}).

For $\|u^{n+1}\|_{\dot H^1}$, we have
\begin{align*}
\| u^{n+1} \|_{\dot H^1} &\le \| u^n\|_{\dot H^1} + \frac{1}A
\| (u^n)^3 - u^n \|_{\dot H^1} \notag \\
& \le (1+\frac{1} A + \frac 3 A \|u^n\|_{\infty}^2 ) \cdot \| u^n \|_{\dot H^1}.
\end{align*}
This completes the proof of Lemma \ref{z1}.
\end{proof}

\begin{lem} \label{z2}
For any $n\ge 0$,
\begin{align}
 &E_{n+1}-E_n 
 +\left(A+\frac 12+\sqrt{\frac{2\nu}{\tau}} \right) \|u^{n+1}-u^n\|_2^2 \notag \\
 \le &\; \|u^{n+1}-u^n \|_2^2 \cdot \Bigl(  \|u^n\|_{\infty}^2 +
 \frac 12 \| u^{n+1} \|_{\infty}^2 \Bigr).
 \end{align}

\end{lem}

\begin{proof}
In this proof we denote by $(\cdot,\cdot)$ the usual $L^2$ inner product.
Recall
\begin{align*}
\frac{u^{n+1}-u^n}{\tau} = -\nu \Delta^2 u^{n+1}
+A \Delta (u^{n+1}-u^n) +\Delta \qn f(u^n).
\end{align*}
Taking the $L^2$ inner product with $(-\Delta)^{-1}(u^{n+1}-u^n)$ on both sides and using the identity
\begin{align} \label{aba_1}
b\cdot (b-a) = \frac 12 ( |b|^2-|a|^2+|b-a|^2), \qquad \forall\, a, b\in \mathbb R^d,
\end{align}
we get
\begin{align}
  & \frac 1 {\tau} \| |\nabla|^{-1} (u^{n+1}-u^n) \|_2^2
  + \frac {\nu}2 ( \| \nabla u^{n+1} \|_2^2 -\| \nabla u^n\|_2^2
+ \| \nabla(u^{n+1}-u^n ) \|_2^2 ) \notag \\
& \qquad + A \|u^{n+1}-u^n \|_2^2 = ( \Delta \qn f(u^n), (-\Delta)^{-1} (u^{n+1}-u^n) ).
\end{align}

Since all $u^n$ have Fourier modes supported in $|k|_{\infty}\le N$, we have
\begin{align}
 (\Delta \qn f(u^n), (-\Delta)^{-1}(u^{n+1}-u^n) ) = - (f(u^n), u^{n+1}-u^n).
 \end{align}

By the Fundamental Theorem of Calculus, we have (recall $f=F^{\prime}$)
\begin{align*}
 &F(u^{n+1})- F(u^n) \notag \\
& = f(u^n) (u^{n+1}-u^n) + \int_{u^n}^{u^{n+1}} f^{\prime}(s) (u^{n+1}-s) ds
\notag \\
& = f(u^n) (u^{n+1}-u^n) + \int_{u^n}^{u^{n+1}} (3s^2-1) (u^{n+1}-s) ds \notag \\
& = f(u^n) (u^{n+1}-u^n) + \frac{(u^{n+1}-u^n)^2 }4 \Bigl( 3(u^n)^2+ (u^{n+1})^2+ 2 u^n u^{n+1}-2 \Bigr).
\end{align*}

Thus
\begin{align}
  & \frac 1 {\tau} \| |\nabla|^{-1} (u^{n+1}-u^n) \|_2^2
  + E_{n+1}-E_n
+\frac{\nu}2 \| \nabla(u^{n+1}-u^n ) \|_2^2  + (A+\frac 12) \|u^{n+1}-u^n \|_2^2  \notag \\
& \qquad= \frac 14 ( (u^{n+1}-u^n)^2, 3 (u^n)^2 + (u^{n+1})^2+ 2 u^n u^{n+1} ) \notag \\
&  \qquad \le \|u^{n+1}-u^n\|_2^2 \cdot
{\frac 14 \Bigl( 3\|u^n\|_{\infty}^2 +  \|u^{n+1}\|_{\infty}^2+ 2
\|u^n\|_{\infty} \|u^{n+1}\|_{\infty} \Bigr)} \notag \\
&\qquad  \le \|u^{n+1}-u^n\|^2_2\cdot \Bigl( \|u^n\|_{\infty}^2 + \frac 12 \|u^{n+1}\|_{\infty}^2 \Bigr). \label{z2_e1a}
\end{align}

Finally observe
\begin{align*}
&\frac 1 {\tau} \| |\nabla|^{-1} (u^{n+1}-u^n) \|_2^2 + \frac{\nu}2 \| \nabla(u^{n+1}-u^n ) \|_2^2 \notag \\
\ge\; & \sqrt{\frac{2\nu} {\tau} } \| |\nabla|^{-1}(u^{n+1}-u^n) \|_2 \| \nabla (u^{n+1}-u^n) \|_2
\ge \sqrt{\frac{2\nu}{\tau} } \| u^{n+1}-u^n\|_2^2.
\end{align*}

The desired inequality then follows easily. In the last step we used Lemma \ref{lem_simple_interp1}.

\end{proof}

\begin{rem}
By using the auxiliary function $g(s)=F(u^n+s(u^{n+1}-u^n))$ and the Taylor expansion
\begin{align*}
g(1)=g(0)+g^{\prime}(0) + \int_0^1 g^{\prime\prime}(s) (1-s)ds,
\end{align*}
we get
\begin{align*}
F(u^{n+1}) &=F(u^n) +f(u^n) (u^{n+1}-u^n) -\frac 1 2 (u^{n+1}-u^n)^2 \notag \\
& \qquad + (u^{n+1}-u^n)^2 \int_0^1 \tilde f^{\prime}( u^n+s(u^{n+1}-u^n) )  (1-s)ds,
\end{align*}
where $\tilde f(z)=z^3$ and $\tilde f^{\prime}(z)=3z^2$ (for $z\in \mathbb R$). From this it is easy to see
that
\begin{align*}
\text{left-hand side of \eqref{z2_e1a} }\le \|u^{n+1}-u^n\|_2^2 \cdot \frac 32
\max\{\|u^n\|_{\infty}^2,\|u^{n+1}\|_{\infty}^2\}.
\end{align*}
This bound will also suffice.
\end{rem}

\subsection*{Proof of Theorem \ref{thm_es_1}}
We inductively prove for all $n\ge 1$,
\begin{align}
& E_n \le E_0, \\
& \|u^n\|_{H^{\frac 32}} \le c_1 \cdot \Bigl(\frac{A+1}{\nu}+\frac 1{A\tau} \Bigr) \cdot (E_0+1),
\end{align}
where $c_1>0$ is the same absolute constant as in Lemma \ref{z1}.

We proceed in two steps. In Step 1 below, we first verify that if the statement holds
for some $n\ge 1$, then it holds for $n+1$. In Step 2, we check the ``base'' case, namely
 for $n=1$ the statement holds. We organize our whole argument in this reverse order (rather than checking the
 base case $n=1$ first and then performing induction) because the verification
 for the base case $n=1$ can be viewed as more or less a special case of the proof in Step 1.

\texttt{Step 1}: the induction step $n \Rightarrow n+1$. Assume the induction
holds for some $n\ge 1$.  We now
verify the statement for $n+1$.

By Lemma \ref{z1}, we have
\begin{align*}
\| u^{n+1}\|_{H^{\frac 32}} \le c_1 \cdot \Bigl(\frac{A+1} {\nu}+\frac 1{A\tau} \Bigr) \cdot (E_n+1)
\le c_1 \cdot \Bigl(\frac{A+1}{\nu}+\frac 1 {A\tau} \Bigr) \cdot (E_0+1).
\end{align*}

Thus we only need to check $E_{n+1} \le E_0$. In fact we shall show
$E_{n+1} \le E_n$.

By Lemma \ref{z2}, we only need to show the inequality
\begin{align} \label{z3_pf_L1}
 A+\frac 12 +\sqrt{\frac{2\nu}{\tau}} \ge
 \|u^n\|_{\infty}^2 +
 \frac 12\| u^{n+1} \|_{\infty}^2.
 \end{align}

We shall use  the log-interpolation inequality (see \eqref{log212} and choose $s=\frac 32$)
for any $f$ with mean zero:
\begin{align}
\| f\|_{L^{\infty}(\mathbb T^2)}
\le 1+d_1\cdot \| f \|_{\dot H^1(\mathbb T^2)}
\cdot \log \Bigl( \| f\|_{H^{\frac 32}(\mathbb T^2)} +3 \Bigr),
\label{log_3.24}
\end{align}
where $d_1>0$ is an absolute constant.

In the rest of this proof, to simplify the notation we shall use $X\lesssim_{E_0} Y$ to
denote $X\le C_{E_0} Y$, where $C_{E_0}$ is a constant depending only on $E_0$. Clearly
\begin{align}
\|u^n\|_{\infty} & \le 1+d_1 \| u^n \|_{\dot H^1}  \log \Bigl(\|u^n\|_{H^{\frac 32}} +3 \Bigr) \notag \\
& \le1+ {d_1 \cdot \sqrt{\frac{2E_0}{\nu}} \cdot \log\Bigl(3+ c_1\cdot \left(\frac{A+1}{\nu}+\frac 1 {A\tau} \right)
\cdot(E_0+1) \Bigr)}
\notag \\
& \lesssim_{E_0} \underbrace{\nu^{-\frac 12}(1+\log A + |\log \nu|)}_{=:m_0} + \nu^{-\frac 12} |\log (2+\frac 1 {\tau})|+1.
\end{align}

Here, in the above inequality, if $\tau \gtrsim  1$ then it is not difficult to check
that the $\log (2+\frac 1 {\tau}) $ term is bounded by a constant and  can be absorbed into $m_0$. In the rest of this proof we shall just
assume $0<\tau\ll 1$ without loss of generality. The case $\tau\gtrsim 1$ is similar and even easier.

Now
\begin{align*}
\|u^n\|_{\infty}^2 \lesssim_{E_0} m_0^2 + \nu^{-1} |\log \tau|^2+1.
\end{align*}

By \eqref{log_3.24} and Lemma \ref{z1}, we have (below in the third inequality
we drop $1/A$ since $A\ge 1$)
\begin{align}
\| u^{n+1} \|_{\infty} &  \lesssim 1+\|u^{n+1} \|_{\dot H^1} \log\left( \|u^{n+1}\|_{H^{\frac 32} } +3 \right) \notag \\
& \lesssim 1+(1+\frac 1 A+ \frac{\|u^n\|^2_{\infty}} A ) \|u^n\|_{\dot H^1} \log\left(\|u^{n+1}\|_{H^{\frac 32}} +3\right)
\notag \\
& \lesssim1+ (1+\frac{\|u^n\|^2_{\infty}} A ) \|u^n\|_{\dot H^1} \log\left(\|u^{n+1}\|_{H^{\frac 32}} +3\right)
\notag \\
&\lesssim_{E_0} 1+(1+\frac{ m_0^2+ \nu^{-1} |\log \tau|^2} A) \cdot\left(m_0+\nu^{-\frac 12} |\log \tau| \right) \notag \\
& \lesssim_{E_0} 1+m_0+ \nu^{-\frac 12} |\log \tau| + \frac{m_0^3+\nu^{-\frac 32}|\log \tau|^3} {A} \notag \\
& \lesssim_{E_0} m_0 + \frac{m_0^3} A + 1+ \nu^{-\frac 32} |\log \tau|^3.
\end{align}
Therefore
\begin{align*}
\|u^n\|_{\infty}^2 + \|u^{n+1}\|_{\infty}^2
\lesssim_{E_0} \left( m_0+\frac{m_0^3} A \right)^2 + 1+ \nu^{-3} |\log \tau|^6.
\end{align*}
Therefore to show inequality \eqref{z3_pf_L1}, it suffices to prove
\begin{align} \label{z3_pf_L22}
A+ \sqrt{\frac{\nu}{\tau}} \ge C_{E_0} \cdot\left(  \Bigl(m_0+\frac{m_0^3}A \Bigr)^2 +1 + \nu^{-3} |\log \tau|^6 \right),
\end{align}
where
\begin{align*}
m_0 = \nu^{-\frac 12} (1+\log A+ |\log \nu|).
\end{align*}

Now we discuss two cases.

\texttt{Case 1}: $\sqrt{\frac{\nu}{\tau}} \ge C_{E_0} \nu^{-3} |\log \tau|^6$. In this case we
choose $A$ such that
\begin{align*}
A\gg_{E_0} m_0^2 = \nu^{-1} ( 1+ \log A + |\log \nu|)^2.
\end{align*}
Clearly for $\nu\gtrsim 1$, we just need to choose $A\gg_{E_0} 1$. On the other hand,
for $0<\nu\ll 1$, it suffices to take
\begin{align*}
A =  \beta \cdot  \nu^{-1} |\log \nu|^2,
\end{align*}
with $\beta$ sufficiently large depending only on $E_0$. Thus in both cases if we
take
\begin{align*}
A = \beta \cdot \max\{ \nu^{-1} |\log \nu|^2, 1 \},
\end{align*}
with $\beta \gg_{E_0} 1$, then
\eqref{z3_pf_L22} holds.

\texttt{Case 2}: $\sqrt{\frac{\nu}{\tau}} \le C_{E_0} \nu^{-3} |\log \tau|^6$. In this case we have
\begin{align*}
|\log \tau| \lesssim_{E_0} 1+|\log \nu|.
\end{align*}
In this case we will not prove \eqref{z3_pf_L22} but prove \eqref{z3_pf_L1} directly.
We first go back to the bound on $\|u^n\|_{\infty}$. It is easy to check that
\begin{align*}
&\|u^n\|_{\infty} \lesssim_{E_0} m_0, \\
&\|u^{n+1}\|_{\infty} \lesssim_{E_0} \left(1+ \frac{m_0^2} A \right) m_0.
\end{align*}
The needed inequality on $A$ then takes the form
\begin{align*}
A \ge C_{E_0} \cdot\left(1+m_0+ \frac{m_0^3}A \right)^2.
\end{align*}
Again we only need to choose $A$ such that $A\gg_{E_0} m_0^2$. The same choice of $A$ as in Case 1
(with $\beta$ larger if necessary) works.

Concluding from both cases, we have proved
 the inequality \eqref{z3_pf_L1} holds. This completes the induction step for $n\Rightarrow n+1$.

\texttt{Step 2}: verification of the base step $n=1$. By Lemma \ref{z1} we have
\begin{align*}
\|u^1\|_{H^{\frac 32}} \le c_1 \cdot\Bigl( \frac{A+1}{\nu}+\frac{1}{ A\tau} \Bigr) \cdot ( E_0+1).
\end{align*}
Therefore we only need to check $E_1\le E_0$. This amounts to checking the inequality
\begin{align*}
A+\frac 12 + \sqrt{\frac{2\nu}{\tau}} \ge \|\Pi_N u_0\|_{\infty}^2 + \frac 12 \|u^1\|_{\infty}^2.
\end{align*}
By Lemma \ref{z1},
\begin{align*}
\| u^1 \|_{\dot H^1}
& \le (1+\frac 1 A+ \frac 3A \| \Pi_N u_0 \|_{\infty}^2)
\cdot \| u_0 \|_{\dot H^1} \notag \\
& \le (1+\frac 1 A+ \frac 3 A \| \Pi_N u_0\|_{\infty}^2) \cdot \sqrt{\frac{2E_0} {\nu}}.
\end{align*}
Therefore
\begin{align*}
\|u^1\|_{\infty}
&\lesssim1+ \|u^1 \|_{\dot H^1} \log( \|u^1\|_{ H^{\frac 32}} + 3) \notag \\
&\lesssim 1+(1+\frac 1 A+ \frac 3 A \|\Pi_N u_0 \|_{\infty}^2)
\cdot \sqrt{ \frac{2E_0}{\nu}} \cdot
\log \bigl( 3+ c_1(\frac{A+1}{\nu} + \frac 1 {A\tau} ) (E_0+1) \bigr) \notag \\
& \lesssim_{E_0} 1+(1+\frac 1 A + \frac 3A \| \Pi_N u_0 \|_{\infty}^2)
\cdot \nu^{-\frac 12} \cdot ( 1+ \log A + |\log \nu| + |\log \tau|).
\end{align*}
Thus we need to choose $A$ such that
\begin{align*}
&A + \frac 12+ \sqrt{ \frac{2\nu}{\tau}} \ge \| \Pi_N u_0 \|_{\infty}^2 +1\notag \\
& \quad\quad + \tilde{C}_{E_0} \cdot (1+\frac 1 A + \frac 3A \|\Pi_N u_0\|_{\infty}^2)^2
\cdot \nu^{-1}
\cdot (1+\log A + |\log \nu| + |\log \tau| )^2,
\end{align*}
where $\tilde C_{E_0}$ is a constant depending only on $E_0$.

By Sobolev embedding, we have
\begin{align*}
\| \Pi_N u_0 \|_{L^{\infty}(\mathbb T^2)} \lesssim \| \Pi_N u_0\|_{H^2 (\mathbb T^2)} \lesssim
\| u_0 \|_{H^2(\mathbb T^2)}.
\end{align*}

Thus it suffices to take $A$ such that
\begin{align*}
A\gg_{E_0} \|u_0\|_{H^2}^2+ \nu^{-1} |\log \nu|^2+1.
\end{align*}
This completes the proof of Theorem \ref{thm_es_1}.

\subsection{Proof of Theorem \ref{thm_es_2}}
This is similar to the proof of Theorem \ref{thm_es_1}. Therefore we only sketch the needed
modifications. In terms of scaling it is useful to think of $\nabla h^n$ as $u^n$ in Theorem
\ref{thm_es_1}.
Write \eqref{semi_e2} as
\begin{align}
h^{n+1} = \frac{1-A \tau \Delta} {1+\nu \tau \Delta^2-A\tau \Delta} h^n
+ \frac{\tau  \qn }{1+\nu \tau \Delta^2 -A\tau \Delta}  \nabla \cdot \bigl( g(\nabla h^n) \bigr).
\end{align}

In place of Lemma \ref{z1} we have the following lemma. We omit the proof since it is quite similar.
\begin{lem} \label{z1p}
There is an absolute constant $c_1>0$ such that
\begin{align}
& \| h^{n+1} \|_{H^{\frac 52}(\mathbb T^2)} \le c_1 \cdot (\frac{A+1}{\nu} +\frac 1 {A\tau})\cdot (E_n+1), \notag \\
& \| h^{n+1} \|_{\dot H^2(\mathbb T^2)}
\le (1+ \frac {1} A + \frac{3}A \|\nabla h^n\|_{\infty}^2)\cdot \|h^n\|_{\dot H^2(\mathbb T^2)}.
 \notag
\end{align}
Here $E_n=E(h^n)$ (see \eqref{MBE_energy}).
\end{lem}

\begin{lem} \label{z2p}
For any $n\ge 0$,
\begin{align}
 &
 E_{n+1}-E_n  + (A+\frac 12+\sqrt{\frac{2\nu}{\tau}}) \|\nabla(h^{n+1}-h^n)\|_2^2 \notag \\
 \le &\; \|\nabla (h^{n+1}-h^n )\|_2^2
 \cdot \frac 32 \max\{ \|\nabla h^n\|_{\infty}^2, \| \nabla h^{n+1} \|_{\infty}^2 \}.
 \end{align}
\end{lem}
\begin{proof}
Taking the inner product with $(h^{n+1}-h^n)$ on both sides of \eqref{semi_e2}, we get
\begin{align*}
   & \frac 1 {\tau} \| h^{n+1}-h^n \|_2^2 + \frac {\nu}2 ( \| \Delta h^{n+1}\|_2^2 -\| \Delta h^n \|_2^2
   + \| \Delta(h^{n+1}-h^n) \|_2^2 ) + A \| \nabla (h^{n+1}-h^n) \|_2^2 \notag \\
   = \; & - (g(\nabla h^n), \nabla (h^{n+1}-h^n) ).
\end{align*}

Recall $g(z)=(|z|^2-1) z=\nabla G$ and $G(z)=\frac 14 (|z|^2-1)^2$. Introduce
\begin{align*}
H(s)= G(\nabla h^n + s(\nabla h^{n+1} -\nabla h^n) ).
\end{align*}

By using the expansion
\begin{align*}
H(1) = H(0) + H^{\prime}(0) + \int_0^1 H^{\prime\prime}(s) (1-s) ds,
\end{align*}
we get
\begin{align*}
G(\nabla h^{n+1})-G(\nabla h^n) &= g(\nabla h^n)\cdot (\nabla h^{n+1}-\nabla h^n) \notag \\
&\quad + \sum_{i,j=1}^2 \partial_i(h^{n+1}-h^n) \partial_j (h^{n+1}-h^n)
\int_0^1 (\partial_{ij} G)(\nabla h^n +s(\nabla h^{n+1}-\nabla h^n) ) (1-s) ds,
\end{align*}

Now denote $\tilde G (z)=\frac 14 |z|^4 $. Then

\begin{align*}
   & E_{n+1}-E_n+ \frac 1 {\tau} \| h^{n+1}-h^n \|_2^2 +\frac {\nu}2
    \| \Delta(h^{n+1}-h^n) \|_2^2  + (A+\frac 12) \| \nabla (h^{n+1}-h^n) \|_2^2 \notag \\
   = \; &
   \sum_{i,j=1}^2 \biggl(\partial_i(h^{n+1}-h^n) \partial_j (h^{n+1}-h^n)
\int_0^1 (\partial_{ij} \tilde G)(\nabla h^n +s(\nabla h^{n+1}-h^n) ) (1-s) ds,\; 1
\biggr),
\end{align*}
where $1$ represents the constant function with value $1$.

Now since $\partial_{ij} \tilde G(z)=|z|^2 \delta_{ij} + 2 z_j z_i$, we have
the point-wise bound $|(\partial_{ij} \tilde G)(z) | \le 3 |z|^2$. Thus
\begin{align*}
\|(\partial_{ij} \tilde G)(\nabla h^n +s(\nabla h^{n+1}-h^n) )\|_{\infty} \le 3 \max \{ \|\nabla h^n \|_{\infty}^2,
\|\nabla h^{n+1} \|_{\infty}^2 \}.
\end{align*}
The desired inequality now follows from this and the simple interpolation inequality (see Lemma \ref{lem_simple_interp1})
\begin{align}
\|\nabla h\|_2 \le \| h\|_2^{\frac 12} \| \Delta h\|_2^{\frac 12}.
\end{align}
This completes the proof of the lemma.
\end{proof}

\subsection*{Proof of Theorem \ref{thm_es_2}}
We only need to check the induction hypothesis
\begin{align*}
& E_n \le E_0, \\
& \|h^n\|_{H^{\frac 52}} \le c_1 \cdot \Bigl(\frac{A+1}{\nu}+\frac 1 {A\tau} \Bigr) \cdot (E_0+1),
\end{align*}
for $n+1$. Here $c_1>0$ is the same absolute constant in Lemma \ref{z1p}.

By Lemma \ref{z1p}, we have
\begin{align*}
\| h^{n+1}\|_{H^{\frac 52}} \le c_1 \cdot \Bigl(\frac{A+1} {\nu}+\frac 1 {A\tau} \Bigr)\cdot (E_n+1)
\le c_1 \cdot \Bigl(\frac{A+1}{\nu}+\frac 1 {A\tau}\Bigr) \cdot (E_0+1).
\end{align*}

Thus we only need to check $E_{n+1} \le E_n$.
By Lemma \ref{z2p}, this amounts to proving the inequality
\begin{align} \label{z3_pf_LLL1}
 A+\frac 12 \ge
\frac 3 2 \max \{ \|\nabla h^n\|_{\infty}^2, \|\nabla h^{n+1}\|_{\infty}^2 \}.
 \end{align}

We shall again use  the inequality
\begin{align}
\| f\|_{L^{\infty}(\mathbb T^2)}
\le 1+ d_1\cdot \| f \|_{\dot H^1(\mathbb T^2)}
\cdot \log\Bigl( \| f\|_{H^{\frac 32}(\mathbb T^2)} +3\Bigr),
\end{align}
where $d_1>0$ is an absolute constant, and $f$ has mean zero.
Clearly
\begin{align}
\|\nabla h^n\|_{\infty} & \le 1+d_1 \| h^n \|_{\dot H^2}  \log(\|h^n\|_{H^{\frac 52}} +3) \notag \\
& \le1+ {d_1 \cdot \sqrt{\frac{2E_0}{\nu}} \cdot \log\Bigl(3+ c_1\cdot
 \left(\frac{A+1}{\nu}+\frac 1{A\tau}\right) \cdot(E_0+1) \Bigr)}.
\end{align}

The rest of the argument now is similar to that in the Proof of Theorem \ref{thm_es_1}. We omit
further repetitive details.

\section{Bounds on the PDE solution of CH}
\setcounter{equation}{0}

Consider
\begin{align}
\begin{cases}
\partial_t w = - \nu \Delta^2 w + \Delta (f (w)), \\
w\Bigr|_{t=0} =w_0.
\end{cases}
\end{align}
Recall that the corresponding energy $E(\cdot)$ is defined by \eqref{CH_energy}.

\begin{prop} \label{bc1}
Let $0<\nu \lesssim 1$. Assume the initial data $w_0 \in H^2(\mathbb T^2)$ with mean zero.
Assume $\| w_0\|_{\infty} \lesssim 1$. Then
\begin{align}
\sup_{0\le t <\infty} \|w(t) \|_{\infty}
\lesssim 1+\sqrt{\frac {E_0} {\nu}}
\cdot \Bigl( |\log \nu| + |\log E_0| +1 \Bigr),
\end{align}
where $E_0= E(w_0)$.
\end{prop}

\begin{proof}
First consider the regime $0<t\ll \nu$. Write
\begin{align*}
w(t) = e^{-\nu t \Delta^2} w_0 + \int_0^t \Delta e^{-\nu (t-s) \Delta^2} f(w(s)) ds.
\end{align*}
Then
\begin{align}
\| w(t) \|_{\infty} & \lesssim \|w_0 \|_{\infty}
+ \int_0^t \nu^{-\frac 12} (t-s)^{-\frac 12} \| f(w(s) )\|_{\infty} ds \notag \\
& \lesssim \| w_0\|_{\infty} + \nu^{-\frac 12} t^{\frac 12}
\cdot \left( \| w\|_{L_{s,x}^{\infty}([0,t])}^3 + \| w\|_{L_{s,x}^{\infty}([0,t])} \right).
\end{align}
By using a continuity argument (on the quantity $\|w\|_{L_{s,x}^{\infty}([0,t])}$), we get
\begin{align}
\sup_{0\le t \le \epsilon_0 \nu} \| w(t) \|_{\infty} \lesssim 1,
\end{align}
where $\epsilon_0>0$ is a sufficiently small absolute constant. (Strictly speaking the
value of $\epsilon_0$ depends on the implied constants hidden in the inequalities $\| w_0\|_{\infty} \lesssim 1$
and $0<\nu \lesssim 1$.)

Next we consider the $L^{\infty}$ bound in the time regime $t\ge \epsilon_0 \nu$. First observe that by using
energy conservation, we have
\begin{align*}
\| \nabla w(t) \|_2 \lesssim \sqrt{\frac {E_0}{\nu}}.
\end{align*}
Set $t_1=t-\frac 12 \epsilon_0 \nu$. Then
\begin{align*}
w(t) = e^{-\nu (t-t_1) \Delta^2} w(t_1)
+ \int_{t_1}^t \Delta e^{-\nu(t-s) \Delta^2} f(w(s))ds.
\end{align*}
We bound the $\dot H^{1+}$-norm of $w$ as
\begin{align}
  \| w(t) \|_{\dot H^{1+}} & \lesssim \| |\nabla|^{1+} e^{-\nu (t-t_1)\Delta^2} w(t_1) \|_2
  + \int_{t_1}^t \| |\nabla|^{3+} e^{-\nu(t-s)\Delta^2} ( w(s)^3 -w(s) ) \|_2 ds \notag \\
  & \lesssim  ( \nu(t-t_1) )^{0-} \| w(t_1) \|_{\dot H^1}
  + \int_{t_1}^t ( \nu(t-s))^{-\frac 34-} ( \| w(s)\|_{\dot H^1}^3 + \| w(s) \|_{\dot H^1} ) ds \notag \\
  & \lesssim \; \nu^{-\frac 12-} \sqrt{E_0} + \nu^{-\frac 34-} \cdot \nu^{\frac 14-}
  \cdot \biggl( \left(\frac{E_0}{\nu}\right)^{\frac 32} + \left(\frac{E_0}{\nu}\right)^{\frac 12} \biggr) \notag \\
  & \lesssim \; \nu^{-1}\cdot \biggl( \left(\frac{E_0}{\nu}\right)^{\frac 32} + \left(\frac{E_0}{\nu}\right)^{\frac 12}
  \biggr).
  \end{align}
Then (recall that $w$ has mean zero)
\begin{align*}
\| w(t) \|_{\infty} & \lesssim 1+\| \nabla w(t) \|_2 \cdot \log ( 10+\|w(t)\|_{H^{1+}} ) \notag \\
& \lesssim \; 1+\sqrt{\frac{E_0}{\nu}} \cdot \Bigl( 1+ |\log \nu| + |\log E_0| \Bigr).
\end{align*}
This completes the proof of Proposition \ref{bc1}.
\end{proof}
\begin{rem}
By using the method in \cite{LQT14}, one can prove a well-posedness result for $w_0\in L^2(\mathbb T^2)$.
However we shall not need this refinement here.
\end{rem}
\begin{prop} \label{bc2}
Assume the initial data $w_0$ have mean zero and $w_0 \in H^s (\mathbb T^2)$, $s\ge 4$. Then
for any $0<\delta\le 1$,
\begin{align}
\int_0^T \| \partial_t \Delta w\|_2^2 dt \lesssim_{\delta, \nu, w_0} 1+T^{\delta}.
\end{align}

\end{prop}

\begin{proof}
To simplify the notation we shall write $\lesssim_{\delta, \nu,w_0}$ as $\lesssim $ throughout this proof.
We shall take $N$ to be a sufficiently large number (it will be clear from the argument
below that $\delta=O(1/N)$).
By using the smoothing effect, it is easy to show that
\begin{align}
\sup_{1\le t<\infty} \| \partial_t w\|_{H^{N}} \lesssim 1.
\end{align}
From energy conservation, we have
\begin{align}
\int_0^{\infty} \| |\nabla|^{-1} \partial_t w\|_2^2 dt \lesssim 1.
\end{align}
By using the interpolation inequality
\begin{align*}
\| \Delta \partial_t w\|_{2} \lesssim
\| |\nabla|^{-1} \partial_t w\|_2^{\frac{N-2}{N+1}}
 \| \langle \nabla \rangle^N \partial_t w \|_2^{\frac 3{N+1}},
 \end{align*}
we get
\begin{align}
\int_1^{\infty} \| \Delta \partial_t w\|_2^{\frac{2(N+1)}{N-2} } dt \lesssim 1.
\end{align}
This implies
\begin{align}
\int_1^{T} \|\Delta \partial_t w\|_2^2 dt \lesssim 1 +T^{\frac 3{N+1}}.
\end{align}
Now we only need to show
\begin{align}
\int_0^1 \| \Delta \partial_t w\|_2^2 dt \lesssim 1.
\end{align}
Observe
\begin{align*}
\partial_t \Delta w = -\nu \Delta^3 w + \Delta^2 f(w).
\end{align*}
Multiplying both sides by $\partial_t \Delta w$ and integrating by parts, we get
\begin{align*}
\| \partial_t \Delta w\|_2^2 = -\frac{\nu}2
\frac d {dt} ( \| \Delta^2 w\|_2^2 ) + \int_{\Omega} \Delta^2 (f(w)) \partial_t \Delta w dx.
\end{align*}
Thus
\begin{align}
\frac{\nu}2 \frac d {dt} ( \| \Delta^2 w\|_2^2 )
&\le - \| \partial_t \Delta w\|_2^2 + \| \Delta^2 f(w) \|_2 \cdot \| \partial_t \Delta w\|_2 \notag \\
& \le -\frac 12 \| \partial_t \Delta w\|_2^2 + \operatorname{const} \cdot
\Bigl( \| w\|_{H^4}^3 + \| w \|_{H^4} \Bigr).
\end{align}
From this (and standard $H^4$ global well-posedness theory), we get
\begin{align}
\int_0^1 \| \partial_t \Delta w\|_2^2 dt \lesssim 1.
\end{align}
The desired inequality then follows.
\end{proof}

\section{Error estimate for CH}
\setcounter{equation}{0}

In this section we give the estimate for $CH$ in $L^2$.
\subsection{Auxiliary $L^2$ error estimate for near solutions}
Consider
\begin{align} \label{ech_e4.1}
\begin{cases}
\displaystyle\frac{v^{n+1}-v^n}{\tau} = - \nu \Delta^2 v^{n+1} +A \Delta(v^{n+1}-v^n) +\Delta \Pi_N f(v^n) +\Delta \tilde
G_n^1, \quad n\ge 0,\\
\displaystyle\frac{\tilde v^{n+1} -\tilde v^n} {\tau} =-\nu \Delta^2 \tilde
v^{n+1} + A \Delta( \tilde v^{n+1}- \tilde v^n) + \Delta \Pi_N  f(\tilde
v^n)+
\Delta \tilde G^n_2, \qquad n\ge 0,\\
v^0=v_0, \quad \tilde v^0=\tilde v_0,
\end{cases}
\end{align}
where $v_0$ and $\tilde v_0$ have mean zero. Denote $\tilde
G^n=\tilde G^n_1-\tilde G^n_2$.

We first state and prove a simple lemma.

\begin{lem}[discrete Gronwall inequality] \label{lemy3}
Let $\tau>0$ and $y_n\ge 0$, $\tilde \alpha_n \ge 0$, $\tilde
\beta_n\ge 0$ for $n=0,1,2,\cdots$. Suppose
\begin{align*}
\frac{y_{n+1}-y_n}{\tau} \le \tilde \alpha_n y_n + \tilde \beta_n,
\quad\forall\, n\ge 0.
\end{align*}
Then for any $m \ge 1$, we have
\begin{align} \label{ech_tmp4.4}
y_m \le \exp\Bigl({\tau \sum_{n=0}^{m-1} \tilde \alpha_n}\Bigr) y_0 + \tau
\sum_{k=0}^{m-1} \exp\Bigl({\tau \sum_{j=k+1}^{m-1} \tilde \alpha_j }\Bigr) \tilde
\beta_k.
\end{align}
In particular
\begin{align}
y_m \le \exp\Bigl({\tau \sum_{n=0}^{m-1} \tilde \alpha_n }\Bigr) ( y_0 + \tau
\sum_{k=0}^{m-1} \tilde \beta_k ).
\end{align}
\end{lem}
\begin{proof}
Clearly
\begin{align*}
y_{n+1} \le (1+\tilde \alpha_n \tau) y_n +\tau \tilde \beta_n \le
e^{\tau \tilde \alpha_n} y_n + \tau \tilde \beta_n, \quad \forall\,
n\ge 0.
\end{align*}
Thus
\begin{align*}
\exp\Bigl({-\tau \sum_{j=0}^n \tilde \alpha_j} \Bigr) y_{n+1} \le \exp\Bigl({-\tau
\sum_{j=0}^{n-1} \tilde \alpha_j} \Bigr) y_n + \tau \exp \Bigl( {-\tau \sum_{j=0}^n
\tilde \alpha_j} \Bigr)\tilde \beta_n.
\end{align*}
Summing $n$ from $0$ to $m-1$, we get
\begin{align*}
\exp\Bigl({-\tau \sum_{j=0}^{m-1}  \tilde \alpha_j}\Bigr) y_{m} \le y_0 + \tau
\sum_{n=0}^{m-1} \exp\Bigl( {-\tau \sum_{j=0}^{n} \tilde \alpha_j} \Bigr) \tilde
\beta_n.
\end{align*}
Thus \eqref{ech_tmp4.4} is obtained.

\end{proof}

\begin{prop} \label{propy2}
For solutions of \eqref{ech_e4.1}, assume for some $N_1>0$, $N_2>0$,
\begin{align}
&\sup_{n\ge 0} \| \tilde v^n \|_{\infty} \le N_1,
\quad \sup_{n\ge 0} \| \nabla v^n \|_2 \le N_2,
 \quad
\sup_{n\ge 0} \| \nabla \tilde v^n \|_{2} \le N_2.
\end{align}
Then for any $m\ge 1$,
\begin{align}
& \| v^m -\tilde v^m \|_2^2 \notag \\
\le &\; \exp\Bigl({m\tau\cdot
\frac{C_1\cdot(1+N_1^4+N_2^4)}{\nu}}\Bigr)\cdot \Bigl( \|  v_0-\tilde
v_0\|_2^2 + A\tau \| \nabla (v_0-\tilde v_0)\|_2^2+
\frac{4\tau}{\nu} \sum_{n=0}^{m-1} \| \tilde G^n \|_2^2
\Bigr),
\end{align}
where $C_1>0$ is an absolute constant.
\end{prop}

\begin{rem}
The same proposition holds if $\Pi_N$ is replaced by the identity operator.
\end{rem}

\subsection*{Proof of Proposition \ref{propy2}}
Denote $e^n= v^n - \tilde v^n$. Then
\begin{align}
\frac{e^{n+1}-e^n}{\tau} = -\nu \Delta^2 e^{n+1}
+A\Delta(e^{n+1}-e^n) +\Delta \Pi_N(f(v^n)-f(\tilde v^n)) +\Delta \tilde
G^n.
\end{align}

Taking the $L^2$ inner product with $ e^{n+1}$ on both sides, we get
\begin{align}
 & \frac 1 {2\tau} ( \| e^{n+1} \|_2^2 -
 \|  e^n \|_2^2 + \| e^{n+1}-e^n\|_2^2 ) \notag \\
 & \qquad + \nu \|\Delta e^{n+1} \|_2^2 + \frac A 2
 ( \| \nabla e^{n+1}\|_2^2 - \| \nabla e^n \|_2^2 + \| \nabla (e^{n+1}-e^n) \|_2^2 ) \notag \\
 =& \; (\tilde G^n, \Delta e^{n+1})
  + (f(v^n)-f(\tilde v^n), \Delta \Pi_N e^{n+1} ). \label{yy30}
 \end{align}

Obviously
\begin{align*}
|(\tilde G^n, \Delta e^{n+1})|\le \frac{ 2\| \tilde G^n\|_2^2} {\nu} + \frac{\nu} 8 \| \Delta e^{n+1} \|_2^2.
\end{align*}

On the other hand, recalling $f^{\prime}(z)=3z^2-1$, we get
\begin{align*}
f(v^n) -f (\tilde v^n) & = \int_0^1 f^{\prime}(\tilde v^n +se^n) ds e^n \notag \\
& = (a_1+a_2 (\tilde v^n)^2 )e^n + a_3 \tilde v^n (e^n)^2+ a_4 (e^n)^3,
\end{align*}
where $a_i$, $i=1,\cdots,4$ are constants which  can be computed explicitly.

We now estimate the contribution of each term. In the rest of this proof, to ease the notation,
we shall denote by $C$ an absolute constant whose value may change from line to line.
Clearly
\begin{align}
&|((a_1+a_2 (\tilde v^n)^2)e^n, \Delta e^{n+1})| \notag \\
 \le &\; C
\cdot (1+\|\tilde v^n\|_{\infty}^2) \| e^n\|_2 \cdot \|\Delta e^{n+1}\|_2
\le \frac{C (1+N_1^4)}{\nu} \|e^n\|_2^2 + \frac{\nu} 8 \| \Delta e^{n+1}\|_2^2.
\end{align}
By using the interpolation inequality $\|f\|_4 \lesssim \| f\|_2^{\frac 12} \| \nabla f\|_2^{\frac 12}$, we
get
\begin{align}
&|(a_3 \tilde v^n (e^n)^2, \Delta e^{n+1})| \notag \\
\le &\; C\|\tilde v^n\|_{\infty} \cdot \|e^n\|_4^2 \cdot \|\Delta e^{n+1}\|_2
\le\; C \cdot \frac{N_1^2 \cdot \|e^n\|_4^4}{\nu} + \frac{\nu} 8 \|\Delta e^{n+1}\|_2^2 \notag \\
\le& \; C \cdot \frac{N_1^2}{\nu} \|\nabla e^n\|_2^2 \| e^n\|_2^2 + \frac{\nu} 8 \| \Delta e^{n+1}\|_2^2
 \le C \frac{N_1^2 N_2^2}{\nu} \|e^n\|_2^2 + \frac{\nu} 8 \| \Delta e^{n+1}\|_2^2.
\end{align}

Similarly
\begin{align}
&|(a_4 (e^n)^3, \Delta e^{n+1} ) | \notag \\
 \le & \; \frac C {\nu} \| e^n\|_6^6+ \frac{\nu} 8 \| \Delta e^{n+1}\|_2^2
 \le \frac C {\nu} \| e^n\|_2^2 \| \nabla e^n\|_2^4 + \frac{\nu} 8 \| \Delta e^{n+1}\|_2^2
\le C \frac{N_2^4}{\nu} \|e^n\|_2^2 + \frac{\nu} 8 \| \Delta e^{n+1}\|_2^2.
\end{align}

Collecting the estimates, we get
\begin{align}
&\frac{\|e^{n+1}\|_2^2 - \| e^n\|_2^2}{\tau} +  A  ( \| \nabla e^{n+1}\|_2^2 -\|\nabla e^n\|_2^2) \notag \\
\le&\; \frac {4}\nu \|\tilde G^n\|_2^2 + C \frac{1+N_1^4+N_2^4}{\nu} \|e^n\|_2^2.
\end{align}

Define
\begin{align*}
y_n =\|e^n\|_2^2 + A\tau \|\nabla e^n\|_2^2, \quad \tilde \alpha= C \frac{1+N_1^4+N_2^4} {\nu},
\quad \tilde \beta_n= \frac{4}{\nu}\|\tilde G^n\|_2^2.
\end{align*}

Then obviously
\begin{align*}
\frac{y_{n+1}-y_n}{\tau} \le \tilde \alpha y_n +\tilde \beta_n.
\end{align*}

The desired result then follows from Lemma \ref{lemy3}.
\qed

\subsection{$L^2$ error estimate for CH (proof of Theorem \ref{thm_CH_L2})}
In this proof to simplify the notation, we shall denote by $C$  a constant depending only
on $(\nu,u_0)$. The value of $C$ may vary from line to line. For any two quantities
$X$ and $Y$, we shall write $X\lesssim Y$ if $X\le CY$. Note that we shall still keep track of
the dependence on the parameter $A$ and also the regularity index $s$.

We need to consider
\begin{align}
\begin{cases}
\displaystyle\frac{ u^{n+1}- u^n} {\tau}
= - \nu \Delta^2  u^{n+1} + A \Delta ( u^{n+1}- u^n)
+\Delta \Pi_N f( u^n), \\
\partial_t u = -\nu \Delta^2 u + \Delta f(u), \\
\tilde u^0 =\Pi_N u_0, \quad u(0)=u_0.
\end{cases}
\end{align}

We first rewrite the PDE solution $u$ in the discretized form.
Note that for a one-variable function $h=h(t)$, we have the formula
\begin{align}
\frac 1 {\tau} \int_{t_n}^{t_{n+1}} h(t) dt & = h(t_n) +
\frac 1 {\tau} \int_{t_n}^{t_{n+1}} h^{\prime}(t) \cdot(t_{n+1}-t) dt, \\
\frac 1 {\tau} \int_{t_n}^{t_{n+1}} h(t) dt & = h(t_{n+1} ) + \frac 1 {\tau} \int_{t_n}^{t_{n+1}} h^{\prime}(t) \cdot(t_{n}-t) dt.
\end{align}

By using the above formula and integrating the PDE for $u$ on the time interval $[t_n,t_{n+1}]$, we
get
\begin{align}
&\frac{u(t_{n+1})-u(t_n)} {\tau} \notag \\
= &\;-\nu \Delta^2 u(t_{n+1}) +A \Delta( u(t_{n+1}) -u(t_n) ) +\Delta \Pi_N  f(u(t_n)) +
\Delta \Pi_{>N} f(u(t_n))+
\Delta
\tilde G^n,
\end{align}
where  $\Pi_{>N}=\operatorname{Id}-\Pi_N$ ($\operatorname{Id}$ is the identity operator) and
\begin{align}
\tilde G^n = -\frac{\nu}{\tau} \int_{t_n}^{t_{n+1}}
\partial_t \Delta u \cdot (t_n-t) dt + \frac 1 {\tau} \int_{t_n}^{t_{n+1}}
\partial_t (f(u)) \cdot(t_{n+1}-t) dt - A \int_{t_n}^{t_{n+1}} \partial_t u dt.
\end{align}
Now
\begin{align}
\| \tilde G^n\|_2
\le \nu \int_{t_n}^{t_{n+1}}
\| \partial_t \Delta u\|_2 dt
+ \int_{t_n}^{t_{n+1}}
\| \partial_t u \|_2 dt (\|f^{\prime}(u)\|_{L_t^{\infty} L_x^{\infty} } +A).
\end{align}
By  Proposition \ref{bc1}, we have $\|u\|_{\infty} \lesssim 1$. Since
$\|\partial_t u \|_2 \lesssim \| \Delta \partial_t u\|_2$ (recall $\partial_t u$ has mean zero), we get
\begin{align*}
\| \tilde G^n\|_2 & \lesssim  (1+A)\int_{t_n}^{t_{n+1}} \| \partial_t \Delta u \|_2 dt \notag \\
& \lesssim (1+A) \cdot \left( \int_{t_n}^{t_{n+1}} \|\partial_t \Delta u\|_2^2 dt \right)^{\frac 12} \cdot \sqrt{\tau}.
\end{align*}
Therefore by Proposition \ref{bc2},
\begin{align*}
\sum_{n=0}^{m-1} \| \tilde G^n\|_2^2
& \lesssim  (1+A)^2 \tau \int_0^{t_m} \| \partial_t \Delta u\|_2^2 dt \lesssim
(1+A)^2\tau \cdot (1+t_m).
\end{align*}
It is easy to check that
\begin{align*}
\sup_{t\ge 0} \| u(t)\|_{H^s} \lesssim_s 1
\end{align*}
which implies
\begin{align}
\sup_{t\ge 0}\| f(u(t) )\|_{H^s} \lesssim_s 1.
\end{align}
This gives
\begin{align}
\sum_{n=0}^{m-1} \| \Pi_{>N} f(u(t_n))\|_2^2 \lesssim_s \; N^{-2s} t_m /\tau.
\end{align}
Therefore
\begin{align}
\tau \sum_{n=0}^{m-1}(\|\tilde G^n\|_2^2 + \| \Pi_{>N} f(u(t_n) )\|_2^2)
\lesssim_s (1+t_m) (\tau^2+ N^{-2s}) (1+A)^2.
\end{align}
Note that
\begin{align*}
\| u_0-\Pi_N u_0\|_2 \lesssim_s N^{-s}, \quad \| \nabla u_0 -\nabla \Pi_N u_0\|_2 \lesssim_s N^{-(s-1)}.
\end{align*}
By Theorem \ref{thm_es_1}, we have
\begin{align*}
\sup_{n\ge 0} \| \nabla u^n \|_2 \lesssim 1.
\end{align*}
Also recall the PDE solution $\sup_{n\ge 0} \|u(t_n)\|_{H^s} \lesssim 1$. Thus
by Proposition \ref{propy2}, we  get
\begin{align*}
\|  u^m -u(t_m) \|_2^2 \lesssim_s (1+A)^2e^{Ct_m}\biggl(
 N^{-2s} + \tau \cdot N^{-2(s-1)}+(1+t_m) (\tau^2+N^{-2s}) \biggr).
\end{align*}
Since by assumption we have $s\ge 4$, clearly by the Cauchy-Schwartz inequality
\begin{align*}
\tau \cdot N^{-2(s-1)} \lesssim \tau^2 + N^{-4(s-1)} \lesssim \tau^2 +N^{-2s}.
\end{align*}
This implies
\begin{align*}
\| u^m-u(t_m) \|_2 \lesssim_s (1+A)e^{Ct_m}( N^{-s} +\tau).
\end{align*}
\begin{rem}
From the above analysis, it is clear that our regularity assumption $H^s$, $s\ge 4$, on the initial data comes
from bounding the term
\begin{align*}
\int \| \partial_t \Delta u\|_2^2 dt
\end{align*}
which in turn arose from rewriting the diffusion term $-\nu \Delta^2 u$ into the time-discretized form.
Recall $\partial_t u = -\nu \Delta^2 u+ \Delta (f(u))$. For $0<t\ll 1$, the linear effect is dominant and one
can roughly regard $\partial_t u \sim \Delta^2 P_{<t^{-\frac 14}} u$, where $P_{<t^{-\frac 14}}$ is the Littlewood-Paley
projection to the frequency regime $|\xi| \lesssim t^{-\frac 14}$. Heuristically speaking
\begin{align*}
\| \partial_t \Delta u \|_2^2 \sim (t^{-\frac 12} \| P_{<t^{-\frac 14}} \Delta^2 u \|_2)^2 \sim t^{-1}
\| P_{<t^{-\frac 14}} \Delta^2 u \|_2^2
\end{align*}
which is barely non-integrable in $t$, provided we assume $H^4$ regularity on $u$. Of course a well-known technique in these situations is to use the maximal
regularity estimates of the linear semigroup to get integrability in $t$. In the $L^2$ case the usual energy estimate suffices and this
is why we need $H^4$ regularity on the initial data.

\end{rem}

\section{Error estimate for MBE}
\setcounter{equation}{0}

\subsection{Auxiliary $H^1$ estimate for MBE}
For MBE we need to consider
\begin{align}
\begin{cases}
\displaystyle\frac{q^{n+1}-q^n}{\tau}
= - \nu \Delta^2 q^{n+1} +A \Delta (q^{n+1}-q^n) + \nabla \cdot \Pi_N (g(\nabla q^n)) +\Delta \tilde G^n_1, \\
\displaystyle\frac{\tilde q^{n+1}-\tilde q^n} {\tau}
=-\nu \Delta^2 \tilde q^{n+1}
+A \Delta(\tilde q^{n+1}-\tilde q^n) + \nabla  \cdot \Pi_N (g (\nabla \tilde q^n))
+\Delta \tilde G^n_2, \\
q^0=q_0, \quad \tilde q^0 =\tilde q_0,
\end{cases}
\end{align}
where we recall $g(z)=(|z|^2-1) z$ for $z\in \mathbb R^2$. As before $q_0$ and $\tilde q_0$ are assumed
to have mean zero. Denote $\tilde G^n=\tilde G^n_1-\tilde G^n_2$.

\begin{prop} \label{propy2_MBE}
Assume for some $N_1>0$
\begin{align}
\sup_{n\ge 0} \Bigl(\| \nabla \tilde q^n \|_{\infty} + \|\Delta \tilde q^n\|_2+\|\Delta q^n\|_2 \Bigr) \le N_1.
\end{align}
Then for any $m\ge 1$,
\begin{align}
\| \nabla(q^n -\tilde q^n) \|_2^2
\le e^{m\tau \cdot \frac{C_1\cdot(1+N_1^4)}{\nu}}
\Bigl( \|\nabla( q^0-\tilde q^0)\|_2^2 + A\tau \| \Delta(q^0-\tilde q^0)\|_2^2+
\frac{2\tau}{\nu} \sum_{n=0}^{m-1}
\| \nabla \tilde G^n \|_2^2 \Bigr),
\end{align}
where $C_1>0$ is an absolute constant.
\end{prop}

\begin{proof}
Denote $e^n = q^n -\tilde q^n$. Then
\begin{align*}
\frac{e^{n+1}-e^n}{\tau}
= - \nu \Delta^2 e^{n+1}
+A \Delta( e^{n+1}-e^n) + \nabla \cdot \Pi_N ( g(\nabla q^n) -g(\nabla \tilde q^n) ) +\Delta \tilde G^n.
\end{align*}

Taking the $L^2$ inner product with $(-\Delta) e^{n+1}$ on both sides, we get
\begin{align}
&\frac 1 {2\tau} ( \| \nabla e^{n+1}\|_2^2 -\| \nabla e^n\|_2^2 +
\|\nabla( e^{n+1}-e^n )\|_2^2 ) + \nu \| \Delta \nabla e^{n+1} \|_2^2 \notag \\
& \quad + \frac A 2 ( \| \Delta e^{n+1} \|_2^2 - \| \Delta e^n \|_2^2 + \| \Delta (e^{n+1}-e^n) \|_2^2 )
\notag \\
= & \quad \underbrace{ (\nabla \tilde G^n, \Delta \nabla e^{n+1})}_{I_1}
 +\underbrace{(g(\nabla q^n) -g(\nabla \tilde q^n), \, \nabla \Delta \Pi_N e^{n+1} )}_{I_2}. \label{yMBE31}
\end{align}

For the first term on the RHS of \eqref{yMBE31}, we simply bound it as
\begin{align}
|I_1 | \le \frac 1 {\nu} \| \nabla \tilde G^n \|_2^2 + \frac {\nu} 4 \| \Delta \nabla e^{n+1} \|_2^2.
\end{align}

For the second term $I_2$, recalling $g(z)=(|z|^2-1)z$, we have
\begin{align*}
g(\nabla q^n) -g(\nabla \tilde q^n) = O (\partial e^n) + O((\partial \tilde q^n)^2 \cdot \partial e^n)
+O( (\partial \tilde q^n) \cdot (\partial e^n)^2 ) + O( (\partial e^n)^3).
\end{align*}
Then
\begin{align}
\| g(\nabla q^n) -g(\nabla \tilde q^n) \|_2 & \lesssim (1+N_1^2) \| \nabla e^n\|_2
+ N_1 \| \nabla e^n\|_4^2 + \| \nabla e^n\|_6^3 \notag \\
& \lesssim (1+N_1^2) \| \nabla e^n\|_2 + N_1 \|\nabla e^n\|_2 \|\Delta e^n\|_2 + \|\nabla e^n\|_2 \| \Delta e^n\|_2^2
\notag \\
& \lesssim (1+N_1^2) \| \nabla e^n\|_2.
\end{align}
Thus
\begin{align}
|I_2| \le C \cdot \frac{1+N_1^4}{\nu} \| \nabla e^n\|_2^2 + \frac{\nu}2 \| \Delta \nabla e^{n+1}\|_2^2.
\end{align}
We then obtain
\begin{align}
 &\frac{\|\nabla e^{n+1}\|_2^2 -\| \nabla e^n\|_2^2}{\tau} +A \Bigl(\| \Delta e^{n+1}\|_2^2 -\| \Delta e^n\|_2^2\Bigr)
 \notag \\
\le&\; C \cdot \frac{1+N_1^4}{\nu} \| \nabla e^n \|_2^2 + \frac {2} {\nu} \| \nabla \tilde G^n\|_2^2.
\end{align}
The desired result then follows from Lemma \ref{lemy3}.
\end{proof}

\subsection{Proof of Theorem \ref{thm_MBE_L2}}
Similarly to the proof of Theorem \ref{thm_CH_L2}, we need to consider
\begin{align*}
\begin{cases}
\displaystyle\frac{ h^{n+1} - h^n} {\tau}
= - \nu \Delta^2  h^{n+1}
+ A \Delta ( h^{n+1}- h^n) + \nabla \cdot \Pi_N (g(\nabla  h^n) ), \\
\partial_t h = -\nu \Delta^2 h+ \nabla \cdot ( g(\nabla h) ), \\
 h^0=\Pi_N h_0, \quad h(0)=h_0.
\end{cases}
\end{align*}
On the time interval $[t_n,t_{n+1}]$, we have
\begin{align}
\frac{h(t_{n+1})-h(t_n)} {\tau}
&= - \nu \Delta^2 h(t_{n+1})
+ A \Delta (h(t_{n+1})-h(t_n) ) +\nabla \cdot \Pi_N ( g(\nabla h(t_n) )) \notag \\
& \qquad \qquad + \nabla \cdot \Pi_{>N} ( g(\nabla h(t_n) ) ) +\Delta \tilde G^n,
\end{align}
where
\begin{align*}
\tilde G^n = -\frac{\nu}{\tau}
\int_{t_n}^{t_{n+1}}
\partial_t \Delta h \cdot (t_n-t) dt
+ \frac 1 {\tau} \int_{t_n}^{t_{n+1}} \Delta^{-1} \partial_t \nabla \cdot ( g(\nabla h(t) ) ) \cdot (t_{n+1}-t) dt
- A \int_{t_n}^{t_{n+1}} \partial_t hdt.
\end{align*}
Now we only need to verify the estimates:
\begin{align}
& \int_0^T \|\partial_t \nabla \Delta h\|_2^2 dt \lesssim_{\nu,h_0} 1+T, \\
&  \int_0^T \| \Delta^{-1} \nabla \partial_t \nabla \cdot ( g(\nabla h) ) \|_2^2 dt \lesssim_{\nu,h_0} 1+T.
\end{align}
Recall
\begin{align*}
 \partial_t h=-\nu \Delta^2 h +  \nabla \cdot (g(\nabla h)).
\end{align*}
Multiplying both sides by $-\Delta^3 \partial_t h$ and integrating by parts, we get
\begin{align*}
 \| \Delta \nabla \partial_t h \|_2^2 = - \frac{\nu}2 \frac d {dt}( \| \Delta^2\nabla h \|_2^2)
 + \int \Delta \nabla \nabla \cdot (g(\nabla h) ) \cdot \Delta \nabla \partial_t h dx,
 \end{align*}
 and
\begin{align}
\frac {\nu}2 \frac d {dt}
\| \Delta^2  \nabla h\|_2^2
& \le - \| \partial_t \Delta \nabla  h\|_2^2 + \| \Delta \nabla \nabla \cdot (g(\nabla h) ) \|_2 \cdot
\| \partial_t \Delta\nabla  h\|_2 \notag \\
& \le -\frac 12 \| \partial_t \Delta\nabla  h\|_2^2
+ \operatorname{const} \cdot (\|h\|_{H^5}^3 + \| h\|_{H^5}).
\end{align}
This (together with standard local well-posedness theory; cf. \cite{LQT14}
for more refined results) yields
\begin{align*}
\int_0^1 \| \partial_t \Delta \nabla  h\|_2^2 dt \lesssim_{\nu,h_0} 1.
\end{align*}
The smoothing effect gives control for $t\ge 1$. Thus
\begin{align*}
\int_0^T \| \partial_t \Delta \nabla h\|_2^2 dt \lesssim_{\nu,h_0} 1+T.
\end{align*}
For the term $\| \Delta^{-1} \nabla \partial_t \nabla \cdot (g(\nabla h) ) \|_2$, we note that
\begin{align}
 &\| \Delta^{-1} \nabla \nabla \cdot ( \partial_t (g(\nabla h) ) ) \|_2 \notag \\
 \lesssim &\; \| \partial_t ( |\nabla h|^2 \nabla h -\nabla h ) \|_2
\lesssim  ( \| \nabla h \|_{\infty}^2 +1) \| \nabla \partial_t h\|_2
 \lesssim_{\nu,h_0} \| \nabla \partial_t h\|_2.
\end{align}
Thus
\begin{align}
\int_0^T \| \Delta^{-1} \nabla  \partial_t \nabla \cdot (g (\nabla h) ) \|_2^2 dt \lesssim_{\nu,h_0} 1+T.
\end{align}
Finally we get
\begin{align*}
\| \nabla(h(t_m) -\tilde h^m) \|_2 \lesssim (1+A) e^{Ct_m} \cdot (N^{-(s-1)} + \tau).
\end{align*}
The theorem is proved.

\section{Concluding remarks}
\setcounter{equation}{0}

In this work we considered a class of large time-stepping methods for the phase field models such
as the CH equation  and the thin film equation with fourth order dissipation. We analyzed
the representative case (see \eqref{semi_e1} and \eqref{semi_e2}) which is first order in time and
Fourier spectral in space, with a stabilization $O(\Delta t)$ term of the form
\begin{align} \label{conr_e1}
A\Delta(u^{n+1}-u^n).
\end{align}
For $A$ sufficiently large ($A\ge O(\nu^{-1} |\log \nu|^2)$), we proved unconditional energy stability
independent of the time step. The corresponding error analysis is also carried out in full detail
($L^2$ for CH and $H^1$ for MBE). It is worth emphasizing that our analysis does not require any additional
Lipschitz assumption on the nonlinearity, or any a priori bounds on the numerical solution.
It is expected our theoretical framework can be extended in several directions. We discuss a few such
possibilities below the fold.

\begin{itemize}

\item General stabilization techniques. There are a myriad of ways of introducing the stabilization term.
Taking the first order in time methods as an example, instead
of \eqref{conr_e1}, one can consider a more general form
\begin{align}
A B(u^{n+1}-u^n),
\end{align}
where $B$ is a general operator. One example is $B=-\Delta^2$ which is already used in the aforementioned
works \cite{ZCST99,BJL11}. Similarly one can consider $B= -(-\Delta)^s$ ($s>0$ is real) or even a general
pseudo differential operator. It will be interesting to carry out a comparative study of these different stabilization
techniques and identify the corresponding stability regions. Another issue is to investigate the lower
bound on the parameter $A$. In typical numerical simulations the stability is observed to hold for relatively
small values of $A$ (the threshold value exhibits a weak dependence on the time step $\tau$ and the diffusion
coefficient $\nu$; cf. the numerical simulation
results in \cite{HLT07}). This certainly
merits further study and probably one has to fine-tune our analysis with some numerically verifiable bounds.

\item Higher order time-stepping methods. In \cite{XT06}, Xu and Tang considered a second order scheme
for MBE:
\begin{align}
 &\frac{3h^{n+1}-4 h^n + h^{n-1}} {2\tau} + \nu \Delta^2 h^{n+1} \notag \\
 =&\; A \Delta( h^{n+1} -2 h^n+ h^{n-1}) + \nabla \cdot \Pi_N g(\nabla(2h^n-h^{n-1}) ),\quad n\ge 1,
\end{align}
where $h^0$ is the initial condition and $h^1$ is computed by the first order scheme \eqref{semi_e2}.
Here to keep some consistency with our setup we
have added the projection operator $\Pi_N$ in front of the nonlinear term.
This scheme is called BD2/EP2 since it is obtained by combining a second order backward
differentiation (BD2) for the time derivative term and a second order extrapolation
(EP2) for the explicit treatment of the nonlinear term. A similar higher order BD3/EP3 scheme is
also presented in \cite{XT06}. The stability analysis in
\cite{XT06} is conditional in the sense that the choice of $A$ depends on the a priori gradient bound on the numerical
solution. Moreover, quite different from the first order (in time) methods, the energy stability for
higher order methods typically takes
the form
\begin{align}
E(h^n) \le E(h^0) + O(\tau), \qquad n \tau \le T,
\end{align}
where the implied constant in the $O(\tau)$ term usually depends on the time interval $[0,T]$.
In yet other words one cannot achieve strict monotonic decay of energy as in the first order case.
A very natural problem is to extend our analysis to cover these cases. By using our analysis it is also possible
to refine the stability results in \cite{SY10} and remove the Lipschitz assumption on the nonlinearity in
the case of the second order implicit scheme.
For second order semi-implicit schemes it is expected that our method can be extended to prove an unconditional
stability result at least for
time steps which are moderately small. We plan to address these issues in a future publication.

\item General phase field models (possibly) with higher order dissipations. In \cite{CJPWW14}, the authors considered
the sixth order scalar model
\begin{align}
\partial_t u = \Delta( \epsilon^2 \Delta -W^{\prime\prime}(u) +\epsilon^2 \eta)(\epsilon^2 \Delta u-W^{\prime}(u)),
\end{align}
where $W(u)=\frac 14 (u^2-1)^2$ and $\eta>0$ is a given constant. This equation arises in the modeling of
pore formation in functionalized polymers \cite{GJXCP12}. The numerical experiments in \cite{CJPWW14} used implicit
time stepping together with Newton's method at each time step. From our point of view it will be interesting to
use the numerical schemes similar to \eqref{semi_e1} and establish the corresponding stability and error convergence
results. In a similar vein one can also consider the volume-preserving vector CH model in the same paper
(see (7) in \cite{CJPWW14}) and also the nonlinear diffusion model in \cite{BJL11}. Yet another possibility
is to study the model with general \emph{fractional} dissipation which is already mentioned in the introduction
of \cite{LQT14}. Also one can extend our analysis to the phase fields models of two-phase complex fluids (see \cite{SY14} for
a pioneering study in this direction).
In any case a first step in the
analysis is to establish similar results to \cite{LQT14}.

\end{itemize}
The above list is certainly not exhaustive. For example we did not include the analysis of the Allen-Cahn model
which will be quite similar to the CH case from our point of view. To keep the presentation simple we leave out
the case of dimensions $d=1$ and $d=3$ which can be similarly handled. It is a quite interesting problem to extend
our analysis to the model considered in \cite{SB11} where an additional forcing term is present.
One can also consider generalizing the
analysis herein to
finite difference schemes and even some hybrid schemes. In \cite{LQT14p1} we will introduce a completely new
approach to tackle some of these problems. Another direction is to consider  the phase field
models with stochastic noises. One can introduce similar numerical stabilization techniques as in the deterministic
case and prove stability and convergence in these settings.
We plan to investigate these problems in
the future.

\subsection*{Acknowledgment.}
We thank the anonymous referees for very helpful remarks and suggestions.
 D. Li was supported by an Nserc discovery grant.
 The research of Z. Qiao is partially supported by the Hong Kong Research Council GRF
grants 202112, 15302214 and NSFC/RGC Joint Research Scheme N\underline{\;\;}HKBU204/12.
 The research of T. Tang is mainly supported by Hong Kong Research Council GRF Grants
 and Hong Kong Baptist University FRG grants.

\frenchspacing
\bibliographystyle{plain}

\end{document}